\documentclass[preprint,12pt]{elsarticle}
\usepackage[latin1]{inputenc}
\usepackage[T1]{fontenc}
\usepackage{amsmath,amssymb}
\usepackage{amsfonts,amsthm}
\usepackage{url}
\usepackage{geometry}
\usepackage{graphicx}

\newcommand{\Ab}{\mathbb{A}}
\newcommand{\Bb}{\mathbb{B}}
\newcommand{\Cb}{\mathbb{C}}

\newcommand{\Nb}{\mathbb{N}}
\newcommand{\Rb}{\mathbb{R}}
\newcommand{\Zb}{\mathbb{Z}}

\newcommand{\Hs}{\mathsf{H}}
\newcommand{\Ws}{\mathsf{W}}

\newcommand{\Hb}{\textbf{\upshape H}}
\newcommand{\Xb}{X}

\newcommand{\Ac}{\mathcal{A}}
\newcommand{\Bc}{\mathcal{B}}
\newcommand{\Cc}{\mathcal{C}}
\newcommand{\Dc}{\mathcal{D}}
\newcommand{\Fc}{\mathcal{F}}

\newcommand{\Jc}{\mathcal{J}}
\newcommand{\Kc}{\mathcal{K}}
\newcommand{\Lc}{\mathcal{L}}
\newcommand{\Pc}{\mathcal{P}}
\newcommand{\Rc}{\mathcal{R}}

\newcommand{\ah}{{\alpha}}
\newcommand{\bh}{{\beta}}

\newcommand{\ii}{{\textbf{\upshape i}}}

\DeclareMathOperator{\alge}{alg}
\DeclareMathOperator{\BUC}{BUC}
\DeclareMathOperator{\coker}{coker}
\DeclareMathOperator{\im}{im}
\DeclareMathOperator{\ind}{ind}
\DeclareMathOperator{\osc}{osc}

\DeclareMathOperator{\esssup}{esssup}
\DeclareMathOperator{\plim}{\Pc-lim}

\DeclareMathOperator{\PC}{PC}

\DeclareMathOperator{\SO}{SO}
\DeclareMathOperator{\APer}{AP}

\DeclareMathOperator{\rk}{rank}
\DeclareMathOperator{\slim}{\underset{\textit{n}\to\infty}{s-lim \,}}
\DeclareMathOperator{\slimnn}{s-lim \,}

\DeclareMathOperator{\plimn}{\underset{\textit{n}\to\infty}{\Pc-lim \,}}
\DeclareMathOperator{\phlimn}{\underset{\textit{n}\to\infty}{\hat{\Pc}-lim \,}}
\DeclareMathOperator{\sgn}{sgn}
\DeclareMathOperator{\supp}{supp}

\providecommand{\lb}[1]{\Lc(#1)}
\providecommand{\lc}[1]{\Kc(#1)}
\providecommand{\pb}[1]{\Lc(#1,\Pc)}
\providecommand{\pc}[1]{\Kc(#1,\Pc)}

\newtheorem{thm}{Theorem}

\newtheorem{lem}[thm]{Lemma}
\newtheorem{prop}[thm]{Proposition}
\newtheorem{cor}[thm]{Corollary}

\theoremstyle{definition}
\newtheorem{defn}[thm]{Definition}
\newtheorem{rem}[thm]{Remark}

\begin{document}
\begin{frontmatter}



\title{Approximation sequences on Banach spaces: a rich approach}



\author[Mascarenhas]{Helena Mascarenhas}
\ead{hmasc@math.tecnico.ulisboa.pt}

\author[Santos]{Pedro A. Santos}
\ead{pedro.santos@tecnico.ulisboa.pt}

\author[Seidel]{Markus Seidel}
\ead{markus.seidel@fh-zwickau.de}

\address[Mascarenhas]{
Departamento de Matem\'atica, 
Instituto Superior T\'ecnico, 
Universidade de Lisboa, 
Av. Rovisco Pais, 
1049-001 Lisboa, 
Portugal}

\address[Santos]{
Departamento de Matem\'atica, 
Instituto Superior T\'ecnico, 
Universidade de Lisboa, 
Av. Rovisco Pais, 
1049-001 Lisboa, 
Portugal}

\address[Seidel]{
University of Applied Sciences Zwickau, 
Dr.-Friedrichs-Ring 2a, 
08056 Zwickau, 
Germany}

\begin{abstract}
Criteria for the stability of finite sections of a large class of
convolution type operators on $L^p(\Rb)$
are obtained. In this class almost all classical symbols are
permitted, namely operators of multiplication with functions in
$[\PC ,\SO, L^\infty_0]$ and convolution operators (as well as Wiener-Hopf
and Hankel operators) with symbols in $[\PC,\SO,\APer,\BUC]_p$.
We use a simpler and more powerful algebraic technique than all previous works: the application of $\Pc$-theory together with the rich sequences concept and localization.
Beyond stability we study Fredholm theory in sequence algebras.
In particular, formulas for the asymptotic behavior of 
approximation numbers and Fredholm indices are given. 
\end{abstract}


\end{frontmatter}


\section{Introduction}
\label{intro}
Consider an infinite dimensional Banach space $X$. Denote by $\Lc (X)$ the Banach algebra of all bounded linear operators on $X$, and by $\Kc(X)$ the ideal of the compact operators in $\Lc(X)$.

Given $A \in \Lc(X)$, a fundamental problem of (linear) numerical mathematics is to approximate the solution of the operator equation
\begin{equation} \label{eq:i1}
Au = v, \quad \mbox{for given} \; v \in X.
\end{equation}
A standard procedure is to choose a sequence of projections $P_n$ which converge in some sense to the identity operator, and a sequence of operators $A_n : \im P_n \to \im P_n$ such that the $A_n P_n$ converge to $A$, and to replace equation \eqref{eq:i1} with the ``simpler'' equations
\begin{equation} \label{AP}
A_n u_n = P_n v \quad \mbox{for} \; n = 1, 2, \ldots,
\end{equation}
the solutions $u_n$ being sought in $\im P_n$. The crucial question is if this method {\em applies} to $A$, {\em i.e.} if the Equations \eqref{AP} possess unique solutions $u_n$ for every right-hand side $v$ and for every sufficiently large ${n}$, and if these solutions $u_n$ converge
to the solution $u$ of the original Equation \eqref{eq:i1}. It is easy to see that  the  applicability of the method is equivalent to the stability of the sequence $\{A_{n}\}$ by which we mean that there is an $n_{0} \in \Nb$ such that for $n> n_{0}$ the operators $A_{n}$ are invertible and the norms of the inverses are uniformly bounded.
On the other hand, it is well known that the invertibility of $A$ is not enough to guaranty  the applicability of the method or the stability of the sequence. A related problem, which originates from statistical physics where finite-dimensional systems of very large dimension (on the order of $10^{8}$) appear, is to infer properties of the large finite-dimensional operator from its infinite-dimensional counterpart. 

The projections $P_{n}$ are chosen depending on the space $X$, the operator $A$, and the purpose for  studying \eqref{eq:i1}. If $X$ is a Lebesgue function space, typical projections are related to spline approximation methods or the finite section method (FSM). For function spaces on the real line, the projections associated with the FSM are simply the operators $P_n: =\chi_{[-n,n]}I$ of multiplication with the characteristic functions of the interval $[-n,n]$, respectively.

Approximation methods have been studied since the mid 1970s by many authors with the use of the so-called algebraic techniques. The main idea is to embed the approximation sequence $\{A_{n}\}_{n \in \Nb}$ in a certain algebra of approximation sequences, in which the applicability of the method is easily related to an invertibility problem in that algebra. The application of Banach algebra techniques to singular integral and convolution operators of increasing complexity in several directions (from continuous coefficients or symbols, to piecewise continuous, slowly oscillating, almost-periodic and so on;  similarly from the Cauchy singular integral operators, to Wiener-Hopf, then Wiener-Hopf plus Hankel operators, or even algebras generated by such operators) first on the Lebesgue space $L^{2}(\Rb)$ (resulting in C$^{*}$-algebras), then on $L^{p}(\Rb)$ for $1 \leq p \leq \infty$ has led to a great increase in the technical difficulty of the proofs. 

Our main objective in this paper is in a sense to  apply a simpler but more powerful and robust variant of the algebraic technique that is capable of solving the applicability problem with the greatest generality in terms of spaces, symbols and coefficients and with much simpler proofs. 

We will prove the following main result, which shows that for a large class  of operators, the finite section method is applicable if and only if the operator is invertible. In the other cases our method gives the necessary and sufficient conditions for the applicability of the FSM, in terms of the invertibility of an additional collection of operators, each a homomorphic image of the original approximation sequence, which we call {\it snapshots}.

\begin{thm}\label{TIntro}
Let $A$ be an operator with its related finite section sequence $\Ab=\{A_n\}$ belonging to a certain class $\Ac_1$.  Then the Finite Section Method applies to $A$ if and only if all snapshots are invertible in $\Lc(X)$.
\end{thm}

The class $\Ac_1$ of Theorem \ref{TIntro} covers finite section sequences of sums and products of multiplication operators and convolution operators, with highly discontinuous generating functions, including $L^{\infty}$ functions with limits at infinity for the multiplication operators and slowly oscillating, almost periodic and piecewise continuous functions for the convolution operators, which can be present together for the first time.\footnote{see Equation \eqref{EAc}.}
 
We believe that the ideas described here can be used with advantage over the traditional algebraic methods to study spline approximation methods, or approximation methods in many other spaces.

The paper is organized as follows: in the next section we introduce definitions and some basic results for the function and multiplier spaces. We use  in some cases simpler definitions than before, and discuss their relation with the classically defined spaces. We also review the previous results regarding the finite section method for convolution-type operators, and introduce the $\Pc$-framework, which is the basis of the new approach. We finish the section by discussing the classes of operators that can be included in the $\Pc$-framework.
In the third section, which is the core of the paper, we introduce the notion of rich sequences and develop the whole algebraic procedure until proving the main result in several versions, one of them 
even including the flip operator. Finally, in the last section, we compare the new results with the older ones and discuss the change of paradigm which results from the present approach.

\section{Some basic concepts}

We collect  in this section some  definitions, concepts  and base results that are needed in the latter part of the work and prove some new and highly interesting results on the classification of convolution operators. We also give very briefly some pointers to the history of the use of algebraic techniques in approximation methods.

\subsection{Function algebras}
Let's start with the definition of \textit{piecewise continuous}, \textit{slowly oscillating}, \textit{uniformly continuous} and \textit{almost periodic} functions. All function algebras here are considered as subalgebras of $L^{\infty}(\Rb)$. We represent the one-point compactification of the real line by $\dot{\Rb}$.
\begin{itemize}
\item For $\lambda\in\dot{\Rb}$ let $\PC^\lambda$ be set of all functions being 
continuous on $\dot{\Rb}\setminus\{\lambda\}$ and having finite one-sided limits 
at $\lambda$. Similarly, $\SO^\lambda$ is the set of all functions being 
continuous on $\dot{\Rb}\setminus\{\lambda\}$ and slowly oscillating at $\lambda$, i.e.
\begin{align*}
\lim_{x\to +0}\osc(f,\lambda+([-x,-rx]\cup [rx,x]))&=0 \quad\text{if}\quad\lambda\in\Rb\\
\lim_{x\to +\infty}\osc(f,[-x,-rx]\cup [rx,x])&=0 \quad\text{if}\quad\lambda=\infty
\end{align*} 
for every $r\in(0,1)$, where $\osc(f,I):=\esssup\{|f(t)-f(s)|:t,s\in I\}$. 
\item By $\PC$ we denote the smallest closed algebra including all $\PC^\lambda$,
$\lambda\in\dot{\Rb}$, and by $\SO$ we denote the smallest closed algebra including 
all $\SO^\lambda$, $\lambda\in\dot{\Rb}$.
\end{itemize}

The definitions of $\SO^\lambda$ and $\SO$ are the ones used in the works of Y. Karlovich 
(where $\SO$ is denoted there by $\SO^\diamond$, see for instance \cite{KaHe2013}).

The definition of $\PC$ is not the usual one, but it is chosen here for two reasons: 
Firstly it gives the definitions of $\PC$ and $\SO$ in a unified and consistent picture, 
and secondly it is much more convenient for the subsequent proofs. Actually, this 
characterization of $\PC$ is equivalent to the classical definition:
\begin{lem}
The algebra $\PC$ is exactly the algebra  of the functions with 
finite one-sided limits at each point $x\in\dot{\Rb}$.
\end{lem}
\begin{proof}
Since the generators of $\PC$ have finite one-sided limits at each point,
the inclusion ``$\subset$''
is obvious. Conversely, let $f$ be piecewise continuous in this latter sense. For any prescribed 
$\epsilon>0$ the set of points $x\in\dot{\Rb}$ with $|f(x+)-f(x-)| > \epsilon$ is 
finite (Otherwise, if it was infinite, it would have an accumulation point in the
(compact) set $\dot{\Rb}$, at which the one-sided limits cannot exist.) For each 
closed interval $I$ between such two points we proceed as follows: For each 
interior point $x\in I$ define $y_x:=(f(x-)+f(x+))/2$ and choose an open neighborhood of 
$x$ such that $\|f-y_x\|<\epsilon$ on this neighborhood. At the endpoints of $I$ do this 
for $y_x$ being the respective one-sided limit. These open sets cover $I$. 
By compactness we can choose a finite subcovering, and finally let $g_I$ be the 
linear interpolation of all the respective $y_{x_i}$ on $I$ with the respective 
one-sided limits of $f$ at the boundary of $I$. Combining these $g_I$ to one 
function $g$ over $\dot{\Rb}$ we arrive at $\|f-g\|_\infty \leq 2\epsilon$. 
Since $g$ has only finitely many discontinuities, $\epsilon$ is arbitrary and
$\PC$ is closed, we are done.
\end{proof}

By $L^\infty_0$ we abbreviate the set of all functions $a\in L^\infty(\Rb)$ with
\[a(\pm\infty)=\lim_{x\to\pm\infty}a(x)=0.\]
Finally, denote by $\BUC$ the set (actually the Banach algebra) of all bounded and uniformly
continuous functions on $\Rb$, by $C(\overline{\Rb})$ the subalgebra of all continuous
functions $a$ with finite limits $a(-\infty)$ and $a(+\infty)$, and by $C(\dot{\Rb})$ 
the subalgebra of those continuous functions with finite limit $a(-\infty)=a(+\infty)$
at infinity. $\APer$ shall stand for the smallest closed subalgebra which contains all
functions $x\mapsto e^{\ii\lambda x}$, $\lambda\in\Rb$. Its elements are called 
almost periodic functions. Clearly, $C(\dot{\Rb})\subset C(\overline{\Rb})\subset\BUC$
and $\APer\subset\BUC$ as well as $\SO^\infty\subset\BUC$.

\subsection{Multiplier algebras}
Let $F:L^2(\Rb)\to L^2(\Rb)$ denote the Fourier transform
\[(Fu)(x):=\int_\Rb u(t)e^{\ii tx} dt,\quad x\in\Rb,\]
and $F^{-1}:L^2(\Rb)\to L^2(\Rb)$ the inverse of $F$.
A function $a\in L^\infty(\Rb)$ is called a Fourier multiplier on $L^p(\Rb)$, 
$1<p<\infty$,\footnote{The cases $p\in\{1,\infty\}$ have already been completely settled in
\cite{MaSaSe2014}. Therefore we can focus on $1<p<\infty$ in the present paper.} 
if the operator 
\[(W^0(a)u)(x):=(F^{-1}aFu)(x),\quad u\in L^2(\Rb)\cap L^p(\Rb),\]
acts on the dense subset $L^2(\Rb)\cap L^p(\Rb)$ such that 
$\|W^0(a)u\|_{L^p(\Rb)}\leq c_p\|u\|_{L^p(\Rb)}$
with some constant $c_p$ independent of $u$. Then $W^0(a)$ extends to a bounded 
linear operator on $L^p(\Rb)$. This extension will again be denoted by $W^0(a)$ 
and is referred to as the Fourier convolution operator with the symbol function $a$. 

Notice that all convolution operators are \textit{shift invariant}, i.e.
 $V_sW^0(b)V_{-s}=W^0(b)$ for all $s\in\Rb$, with the so-called shift operators
$V_s$ defined by
\[V_s:L^p(\Rb)\to L^p(\Rb), \; f(x)\mapsto f(x-s).\]
By a theorem of H\"ormander (\cite{Hormander1960}), also the converse is true: 
every shift invariant operator $A\in\lb{L^p(\Rb)}$ is a Fourier convolution operator.

The set of all multipliers on $L^p(\Rb)$
\[M^p:= \{a\in L^\infty(\Rb): W^0(a)\in \lb{L^p(\Rb)}\}\]
 is known to be a Banach algebra with the norm
$\|a\|_{M^p}:=\|W^0(a)\|_{\lb{L^p(\Rb)}}$, and particularly includes 
$x\mapsto -\sgn(x)$ for every $1<p<\infty$, the symbol of the Cauchy singular 
integral operator $S:L^p(\Rb)\to L^p(\Rb)$
\[(Su)(x):=\frac{1}{\pi\ii}\int_\Rb\frac{u(y)}{y-x}dy,\quad x\in\Rb.\] 

Let $p\in(1,\infty)\setminus\{2\}$. By $M^{<p>}$ we denote the set of all 
multipliers $b\in M^p$ for which there exists a $\delta>0$ (depending on $b$)
such that $b\in M^r$ for all $r\in (p-\delta,p+\delta)$. 
Also set $M^{<2>}:=M^2=L^\infty(\Rb)$. 
Furthermore, for a subalgebra $\Bc\subset L^\infty(\Rb)$
let $\Bc_p$ denote the closure in $M^p$ of $\Bc\cap M^{<p>}$.

\medskip
This yields e.g. the algebra $\PC_p$ of piecewise continuous (which particularly contains 
$-\sgn$, the symbol of the Cauchy singular integral operator) and the algebra $\SO_p$ 
of slowly oscillating multipliers (which includes the algebra $\SO_p^\diamond$ defined 
in \cite{KaHe2013}).
Actually, we are going to attack a larger algebra within this paper: Let 
$[\PC,\SO,\BUC]$ stand for the smallest closed algebra which includes \textit{all} 
$\PC$-, $\SO$-, and $\BUC$-functions, (hence also all $\APer$-functions),  we are then interested in all multipliers belonging to $[\PC,\SO,\BUC]_p$. 

Notice that the definitions of these algebras are intuitive, simple and short, and
one might ask whether they are equivalent to the classical approach and the more
technical definitions in the literature. 
We have the following result, whose proof is trivial.
\begin{lem}\label{LPCp}
The algebra $\PC_p$ includes the smallest closed subalgebra of $M^p$ that 
is generated by all $C(\dot{\Rb})_p$-functions and all $\chi_{[\lambda,\infty)}$, 
$\lambda\in\Rb$. 
Moreover, it includes the smallest closed subalgebra of $M^p$ that contains all
piecewise constant functions having only finitely many discontinuities.

Also, $\PC_p\supset \PC_p^\lambda$ and $\SO_p\supset\SO_p^\lambda$ for every $\lambda\in\Rb$.
\end{lem}

\begin{rem}
These new definitions suggest the following  open questions:
\begin{enumerate}
\item whether the collection of all $\PC_p^\lambda$ functions is sufficient to generate $\PC_p$,
\item the same question for $\SO_p^\lambda$ and $\SO_p$,
\item whether  $C(\dot{\Rb})_p$ is a proper subset of $C(\dot{\Rb})\cap M^p$,
\item the same questions for $\SO_p$ and $\PC_p$,
\item whether $\SO_p$ is larger than the algebra $\SO_p^\diamond$ defined in \cite{KaHe2013}.
\end{enumerate}
\end{rem}

Anyway, the more elegant definitions in the present paper cover 
(and probably generalize, depending on the answers to the questions above)
all previous ones.

\subsection{The Finite Section Method in algebras of convolution operators}

Let $\Pc:=(P_n)$ be the sequence of projections $P_n=\chi_{[-n,n]}I$ on the space
$L^p(\Rb)$ and consider the 
\textit{Finite Section Method}
\begin{equation} \label{FSM}
P_{n}AP_n u_n = P_n v \quad \text{for} \quad n = 1, 2, \ldots
\end{equation}
for an operator $A$.  

The applicability and properties of the FSM described above, and its discrete analogue in sequence spaces caught the interest of a large number of researchers that have obtained and extended results in various directions since at least the 1960s up to the present: various types of symbols and  operators \cite{B1984,DidSi2002,KMS2010,Li2006,RSS2010a,2RoS1988,Si1981}, multidimensional case \cite{B19842,BS1983,HRS1994,Ko1973b,MS2005},  spline methods \cite{DidSi2002,HRS1994,JMS2010,San2012}, Toeplitz, Hankel and other operators \cite{RRS2001,RoSa2014a,RSS1997}, 
asymptotics, behavior of eigenvalues, approximation numbers and numerical ranges \cite{BaEh2001,BSW1994,ERS2011,Ro2000, RogSi2006,Wid1974}, structure of the approximation algebras \cite{BaToSi2015,DPSS2014,SS2006}, are just some examples. The above references are far from being exhaustive, see also the monographs \cite{BS1999,BS1989/90,HRS2001,PS1991,RRS2004,RSS2010book} and the references cited therein.

Gohberg and Feldman \cite{GohFel1974} were the first to obtain conditions for the applicability of the method for some classes of continuous symbol convolution operators on $L^{p}(\Rb)$, and Kozak \cite{Ko1973} introduced the algebraic approach.
Another important early milestone for tackling piecewise continuous generating functions was Silbermann's paper \cite{Si1981} where he found a compact-like ideal of sequences and a corresponding snapshot, so that the stability problem could be decomposed into the invertibility of the snapshot and invertibility of the coset in the quotient algebra. That idea has been used ever since, usually complemented by localization techniques to find invertibility conditions for the coset.

When $P_n$ is strongly convergent most proofs make use of the close relationship between strong convergence and the ideal of compact operators. In this context, the observation that the multiplication of a strongly convergent sequence by a compact operator gives uniform convergence plays a fundamental role. For non-strongly convergent projections, as it is the case for the finite section projections in $l^{\infty}(\Zb)$ or $L^{\infty}(\Rb)$, S. Roch and B. Silbermann \cite{RoSi1989} (see also \cite[4.36 et seq.]{PS1991})  initiated the idea of changing the above mentioned connection between compactness and strong convergence to a definition. Starting with the projection sequence $\Pc := (P_n)_{n \in \Nb}$, one substitutes the usual compact operators by compact-like operators related to that sequence. The $\Pc$-compact operators will be those operators $K \in \Lc(X)$ such that $\|KP_n - K\|$ and $\|P_nK - K\|$ tend to zero as $n \to \infty$. 

This new approach with the adapted ideal of $\Pc$-compact operators happens to also be useful in the cases $1<p<\infty$, and is at the heart of the results we present. In  \cite{MaSaSe2014} the authors managed to use the new approach to get, besides the complete solution for the cases $p \in \{1, \infty\}$, the solution for the so-called quasi-banded operators, which include all convolutions with generating function continuous at all finite points of $\Rb$, but do not include the general piecewise continuous case. The present work purposes to bring  the $\Pc$-framework together with localization techniques in order to cover for the first time quasi-banded operators and operators of convolution with piecewise continuous generating functions.

\subsection{The $\Pc$-framework and Fredholm property}

Let $\Xb$ be a Banach space and $\Pc=(P_n)_{n\in\Nb}$ a sequence of projections 
$P_n\in\lb{\Xb}$ such that
\begin{enumerate}
\item $P_nP_{n+1}=P_{n+1}P_n=P_n$ for all $n\in\Nb$,
\item $\lim_{n\to\infty}\|P_nx\|=\|x\|$ for each $x\in\Xb$,
\item $\|\sum_{i\in U} (P_{i}-P_{i-1})\|=1$ for every finite subset $U\subset\Nb$.
\end{enumerate}
Although these conditions seem to be somewhat technical at a first glance, they actually
have a very natural meaning: the projections are nested (1.), cover the whole space (2.),	
and provide a uniformly bounded partition (3.), in a sense. Therefore $\Pc$ is said
to be a \textit{uniform approximate identity}. For a rigorous treatment (in an even
more relaxed setup) we refer to \cite{RRS2004, SeSi2012}.
In fact, recall from \cite{RRS2004} that the set
\[\pc{\Xb}:=\{K\in\lb{\Xb}: \|K(I-P_n)\|,\|(I-P_n)K\|\to 0\text{ as }n\to\infty\}\]
of all \textit{$\Pc$-compact operators} is a Banach algebra. Further we
denote the set of all operators which are compatible with $\pc{\Xb}$ by
\[\pb{\Xb}:=\{A\in\lb{\Xb}: AK,KA\in\pc{\Xb} \text{ for every }K\in\pc{\Xb}\}.\]
It is known that $\pb{\Xb}$ is a Banach algebra as well, and $\pc{\Xb}$ forms
a closed two-sided ideal in $\pb{\Xb}$.

\medskip
An operator $A\in\pb{\Xb}$ is said to be \textit{$\Pc$-Fredholm} if there exists a 
$\Pc$-regularizer $B\in\lb{\Xb}$, that is,  $AB-I, BA-I \in\pc{\Xb}$. These $\Pc$-regularizers belong to $\pb{\Xb}$ again (\cite[Theorem 1.16]{SeSi2012}). Therefore one may also equivalently define
that $A$ is $\Pc$-Fredholm if $A+\pc{\Xb}$ is invertible in the quotient algebra
$\pb{\Xb}/\pc{\Xb}$. Of course, the usual Fredholm property has to 
dovetail with this modified setting:
\begin{itemize}
\item $A\in\pb{\Xb}$ is Fredholm iff there exist projections
	 			$P,P' \in \pc{\Xb}$ of finite rank such that 
				$\im P = \ker A$ and $\ker P' = \im A$. In that case $A$ is $\Pc$-Fredholm.
\item $A\in\pb{\Xb}$ is not Fredholm iff for every $\epsilon>0$ and
				every $l\in\Nb$ there exists a projection $Q\in\pc{\Xb}$ with $\rk Q \geq l$ 
				such that $\|AQ\|<\epsilon$ or $\|QA\|<\epsilon.$ 
\end{itemize}
If the above holds then  $\Pc$ is said to equip $\Xb$ with the $\Pc$-dichotomy, and actually 
this is shown to be true in the cases one is usually interested in, such as $\Xb$ being 
a Hilbert space, $\Xb=L^p(\Rb^N)$ or $\Xb=l^p(\Zb^N)$ with $p\in[1,\infty]$, etc.
(see \cite[Sections 1.2, 3.3.1 and Proposition 1.27]{SeSi2012}).

\medskip
We say that a sequence $(A_n)\subset\pb{\Xb}$ \textit{converges $\Pc$-strongly} to 
$A\in\lb{\Xb}$ if 
\[\|K(A_n-A)\| + \|(A_n-A)K\| \to 0 \quad\text{as}\quad n\to\infty\]
holds for every $K\in\pc{\Xb}$.
In that case $(A_n)$ is automatically bounded, its $\Pc$-strong limit
$A=\plim A_n$ is automatically contained in $\pb{\Xb}$ and
\begin{equation}\label{EPSLim}
\|A\|\leq\liminf \|A_n\|.
\end{equation}

These $\Pc$-substitutes for the classical triple (compactness, Fredholmness, strong convergence) provide us with a ``universe'' which mimics
the classical world in large parts. The advantage is that this new framework has a simple consistent and symmetric algebraic structure, it is more flexible and it is better adapted to the setting $\Xb$ with the sequence of the projections $P_n$. 

In the next section, we return to our concrete setting $\Xb=L^p(\Rb)$ with the sequence 
of the (non-compact) canonical projections $P_n$. Notice that this approach already 
served as a key instrument in \cite{MaSaSe2014} for the treatment of the finite sections 
of arbitrary convolution and convolution type operators operators on the spaces 
$L^1(\Rb)$ and $L^\infty(\Rb)$, where the situation with the classical approach is even 
worse: the $P_n$ do not converge $*$-strongly there.

\subsection{Classes of convolution operators in $\pb{\Xb}$}

Consider $\Xb=L^p(\Rb)$, $1<p<\infty$, with $\Pc=(P_n)=(\chi_{[-n,n]}I)$ the sequence 
of the canonical projections and abbreviate the complementary projections by
$Q_n:=I-P_n$. 

Notice that $\lc{\Xb}\subset\pc{\Xb}\subset\pb{\Xb}\subset\lb{\Xb}$ and all 
inclusions are proper (see e.g. \cite{Li2006}).
\begin{defn}
An operator $A\in\lb{\Xb}$ is called band operator with band-width $w$ if 
\[Q_{1+w}V_sAV_{-s}P_1=P_1V_sAV_{-s}Q_{1+w}=0\quad\text{for every $s\in\Rb$.}\] 
The elements in the smallest closed algebra which contains all band operators
are called band-dominated operators.
Moreover, $A\in\lb{\Xb}$ is said to be quasi-banded if 
\begin{equation}\label{EDefQB}
\lim_{m\to\infty}\sup_{n>0}\|Q_{n+m}AP_n\| =
\lim_{m\to\infty}\sup_{n>0}\|P_n AQ_{n+m}\| = 0.
\end{equation}
\end{defn}
Almost obviously, band operators are band-dominated, and band-dominated ones 
are quasi-banded. Moreover, one can show \cite{MaSaSe2014, RRS2004, Li2006} 
that both, the set of band-dominated operators and the set of 
quasi-banded operators are closed and inverse closed subalgebras of
$\pb{\Xb}$ which contain $\pc{\Xb}$ as a closed ideal. If $A$ from one of these 
algebras is $\Pc$-Fredholm, then its $\Pc$-regularizers are in the same algebra 
again.

\begin{rem}
Roughly speaking these definitions determine to what extend the operators are able to
transport information along the axis.
The definition of band operators is primarily motivated by the analogous 
$l^p$-setting, where one has a natural interpretation of bounded linear operators 
as infinite matrices.
Then the definition just means that band operators have only a finite number of
non-zero diagonals. 
\end{rem}

Besides, there are equivalent characterizations which appear to be very natural also 
in the $L^p$-setting. For this, define the commutator
$[A,B] := AB - BA$ of two operators $A, B$, and moreover, for every function $\varphi$
over $\Rb$ and $t>0$ introduce the inflated copy $\varphi_t$ by $\varphi_t(x)=\varphi(x/t)$. 
\begin{prop}\label{PBDOComm}(\cite[Theorem 2.1.6]{RRS2004} and \cite[Theorem 1.42]{Li2006}) 
$A\in\lb{\Xb}$ is band-dominated if and only if
\begin{equation*}
\|[A,\varphi_tI]\|\to0\quad\text{as}\quad t\to\infty\quad\text{for every function}
\quad\varphi\in\BUC. 
\end{equation*} 
Surprisingly this is also equivalent to 
\[\|[A,\varphi_tI]\|\to0\quad\text{as}\quad t\to\infty\quad\text{for the particular functions }
\varphi(x):=e^{\pm\ii x}.\] 
\end{prop}
For convolutions (i.e. shift-invariant operators, by H\"ormanders result) this means
\begin{cor}\label{CQBOComm}
$W^0(b)$ is band-dominated if and only if 
\begin{equation}\label{EBDOComm}
\|[W^0(b),\varphi_tI]\|\to0\quad\text{as}\quad t\to\infty\quad\text{for every function}
\quad\varphi\in\BUC.
\end{equation} 
In this case $b\in\BUC$.
On the other hand, if $b\in\BUC_p$ then $W^0(b)$ is band-dominated.
\end{cor}
\begin{proof} 
The first part is Proposition \ref{PBDOComm}.
For the second, we start with the case $p=2$. Then, with $\varphi(x):=e^{\pm\ii x}$,
\begin{align*}
\|[W^0(b),\varphi_t I]\|&=\|[F^{-1}bF,\varphi_t I]\|=\|F^{-1}[b I,F\varphi_tF^{-1}]F\|\\
&=\|[b I,W^0(\varphi_{-t})]\|=\|[b I,V_{\pm 1/t}]\|,
\end{align*}
which tends to $0$ as $t\to\infty$ iff $b\in\BUC$. On the other hand, again by Proposition 
\ref{PBDOComm}, this tends to zero iff $W^0(b)$ is band-dominated.

If $W^0(b)$ is band-dominated on $L^p(\Rb)$, $p\neq 2$, then it is also band-dominated on 
$L^2(\Rb)$, hence $b\in\BUC$.

Finally, let $b\in\BUC_p$. Since the algebra of band-dominated operators is closed 
we can restrict our considerations to a dense subset and assume that $b\in M^r$ for 
an $r$ with $|r-2|>|p-2|$, by the definition of $\BUC_p$.
Then, for every $\varphi\in\BUC$, $\|\varphi\|_\infty=1$, it holds that
$\|[W^0(b),\varphi_t I]\|_{\lb{L^r(\Rb)}}\leq 2\|W^0(b)\|_{\lb{L^r(\Rb)}}$
for all $t>0$. On the 
other hand $\|[W^0(b),\varphi_t I]\|_{\lb{L^2(\Rb)}}\to 0$ as $t\to\infty$. By the
Riesz-Thorin Interpolation Theorem then also $\|[W^0(b),\varphi_t I]\|_{\lb{L^p(\Rb)}}\to 0$ 
as $t\to\infty$, thus $W^0(b)$ is band-dominated by Equation \eqref{EBDOComm}.
\end{proof}

For quasi-banded convolutions we have a characterization similar to 
\eqref{EBDOComm} as well. 
\begin{thm}\label{TQBCom}
A convolution $W^0(b)\in\lb{\Xb}$ is quasi-banded if and only if
\begin{equation}\label{EQBCom}
\|[W^0(b),\varphi_tI]\|\to 0 \text{ as } t\to \infty 
\quad\text{for every function}\quad \varphi\in C(\overline{\Rb}).
\end{equation}
\end{thm}
\begin{proof}
Define $\chi_-:=\chi_{(-\infty,0]}$ and $\chi_+:=\chi_{[0,\infty)}$ and
recall from \cite[Proposition 13]{MaSaSe2014} that a shift-invariant operator 
$A=W^0(b)\in\lb{\Xb}$ is quasi-banded if and only if $\chi_\pm A \chi_\mp I$ are $\Pc$-compact.
Let \eqref{EQBCom} be true. The desired $\Pc$-compactness easily follows if 
$\|\chi_\pm A V_{\mp n}\chi_\mp V_{\pm n}\|\to 0$ as $n\to\infty$, since $A$ is 
shift invariant. But this is obvious taking \eqref{EQBCom} for the continuous
piecewise linear splines $\varphi(x)=0$ $(x\leq0)$, $\varphi(x)=1$ $(x\geq1)$ 
resp. $1-\varphi$ into account, e.g.
\[\|\chi_- A V_{n}\chi_+ V_{-n}\|=\|\chi_- A \varphi_{n} V_{n}\chi_+ V_{-n}\|
\approx\|\chi_- \varphi_{n} A V_{n}\chi_+ V_{-n}\|=0.\]

Conversely, let $A=W^0(b)$ be quasi-banded, $\varphi\in C(\overline{\Rb})$, 
$\|\varphi\|_\infty=1$, and $\epsilon>0$. Since $A-\chi_-A\chi_-I-\chi_+A\chi_+I$
is $\Pc$-compact for which \eqref{EQBCom} obviously holds, it suffices to consider 
$B=\chi_+A\chi_+I$.
Firstly choose $m\in\Nb$ such that $|\varphi(x)-\varphi(+\infty)|<\epsilon$ for all 
$x\geq m$. Moreover, there is an $r$ such that on each of the intervals
$[(k-1)m/r,km/r]$, $k=1,\ldots,r$, it holds that $|\varphi(x)-\varphi(km/r)|<\epsilon$,
respectively. W.l.o.g. we can assume that $m/r=1$, otherwise replace $\varphi$ by
$\varphi_{r/m}$. Also, there exists a $t_0$ such that 
$\|\varphi_t-\varphi_{\left\lfloor t\right\rfloor}\|_\infty<\epsilon$
for every $t\geq t_0$, where $\left\lfloor t\right\rfloor$ denotes the largest
integer $\leq t$.
Since $B$ is quasi-banded we can find a $t_1\in\Nb$, $t_1\geq t_0$ such that, 
with $Q_n:=I-P_n$,
\begin{equation*}
\sup_{n>0}\|Q_{n+t}BP_n\| <\frac{\epsilon}{2m+1}\text{ and } 
\sup_{n>0}\|P_n BQ_{n+t}\| <\frac{\epsilon}{2m+1}
\text{ for all $t\geq t_1$.}
\end{equation*}
Fix $t\in\Nb$, $t\geq t_1$.
For $i\in\Nb$ set $P_{\{i\}}:=P_i-P_{i-1}$ and for $U\subset\Nb$  
introduce the notation \mbox{$P_U:=\sum_{i\in U}P_{\{i\}}$} and $Q_{U}:=I-P_{U}$.
Now set $U_0:=\emptyset$, $U_k:=\{(k-1)t+1,\ldots,kt\}$,
$U_{m+1}:=\{mt+1,\ldots\}$ and $V_k:=U_{k-1}\cup U_k\cup U_{k+1}$,
$V_{m+1}:=\{(m-1)t+1,\ldots\}$ for all $k=1,\ldots,m$.
Then
\begin{align*}
\sum_{k=1}^{m+1}\|P_{U_k}BQ_{V_k}\| 
&\leq \sum_{k=1}^{m}\left(\|P_{U_k}P_{kt}BQ_{(k+1)t}Q_{V_k}\|
+\|P_{U_k}Q_{(k-1)t}BP_{(k-2)t}Q_{V_k}\|\right)\\
&+\|Q_{mt}BP_{(m-1)t}\| < \epsilon.
\end{align*}
Now we apply the partition $\Nb=\bigcup_{k=1}^{m+1} U_k$ as follows:
\begin{align*}
B\varphi_tI &=\sum_{k=1}^{m+1} P_{U_k}B\varphi_tI
=\sum_{k=1}^{m+1} P_{U_k}B\varphi_t P_{V_k} + \sum_{k=1}^{m+1} P_{U_k}B Q_{V_k}\varphi_t I,
\end{align*}
where the second term is less than $\epsilon$ as is shown before.
For the first term we abbreviate $\psi_k:=\varphi_t(kt)=\varphi(k)$, 
$k=1,\ldots,m$, and $\psi_{m+1}:=\varphi(+\infty)$, and get
\begin{align*}
\sum_{k=1}^{m+1} P_{U_k}B\varphi_t P_{V_k}
&=\sum_{k=1}^{m+1} P_{U_k}B(\varphi_t-\psi_k) P_{V_k} 
+ \sum_{k=1}^{m+1} \psi_kP_{U_k}B P_{V_k}\\
&=\sum_{k=1}^{m+1} P_{U_k}B(\varphi_t-\psi_k) P_{V_k}
+ \sum_{k=1}^{m+1} (\psi_k-\varphi_t)P_{U_k}B P_{V_k} 
+ \varphi_t\sum_{k=1}^{m+1} P_{U_k}B P_{V_k}\\
&=\sum_{k=1}^{m+1} P_{U_k}B P_{V_k} C_k
+ d_1\sum_{k=1}^{m+1} P_{U_k}B P_{V_k} 
- \varphi_t\sum_{k=1}^{m+1} P_{U_k}B Q_{V_k} + \varphi_t B
\end{align*}
where $C_k=d_2U_{k-1}+d_3U_k+d_4U_{k+1}$ with certain functions 
$d_1,d_2, d_3, d_4\in L^\infty(\Rb)$ of the norm less than $2\epsilon$.
Also the norm of the third term is less than $\epsilon$, again by the above 
construction. Thus, for integers $t\geq t_1$,
\begin{equation}\label{ESum}
\|B\varphi_tI-\varphi_tB\|\leq 8\epsilon\left\|\sum_{k=1}^{m+1} P_{U_k}B P_{V_k}\right\|
+2\epsilon.
\end{equation}
Due to the choice of $t_0$ we get a similar estimate for arbitrary real numbers 
$t\geq t_0$ with an additional summand $2\|B\|\epsilon$ at the right hand side.
It remains to estimate the sum in \eqref{ESum}: 
\[\left\|\sum_{k=1}^{m+1} P_{U_k}B P_{V_k}\right\|\leq
\left\|\sum_{k=2}^{m+1} P_{U_k}B P_{U_{k-1}}\right\|
+\left\|\sum_{k=1}^{m+1}P_{U_k}BP_{U_k}\right\|
+\left\|\sum_{k=1}^{m}P_{U_k}BP_{U_{k+1}}\right\|.\]
For the middle term $\left\|\sum_{k=1}^{m+1} P_{U_k}B P_{U_k}\right\|\leq\|B\|$ 
is obvious since this operator can be regarded as a block diagonal operator. 
Similarly, with the blocks  $P_{U_k}(B V_{-t})P_{U_k}$,
\[\left\|\sum_{k=2}^{m} P_{U_k}B P_{U_{k-1}}\right\|
=\left\|\sum_{k=2}^{m} P_{U_k}B V_{-t}P_{U_{k}}V_{t}\right\|\leq \|B\|,\]
as well as
$\left\|\sum_{k=1}^{m-1}P_{U_k}BP_{U_{k+1}}\right\|\leq\|B\|$. 
Finally, the two remaining summands have also norm $\|\cdot\|\leq\|B\|$
and we can conclude that the sum in \eqref{ESum} is $\leq 5\|B\|$, hence
\[\|B\varphi_tI-\varphi_tB\|\leq 8\epsilon\cdot5\|B\|+2\epsilon+2\|B\|\epsilon
\leq(42\|A\|+2)\epsilon\quad\text{for all } t\geq t_1.\] 
With a similar estimate for $B=\chi_-A\chi_-I$ this yields \eqref{EQBCom} and 
finishes the proof since $\epsilon>0$ was chosen arbitrarily.
\end{proof}

In \cite{MaSaSe2014} it has already been proved that quasi-banded operators, hence 
most of the operators we are interested in, belong to $\pb{\Xb}$. In particular, 
convolutions with multipliers in $\SO_p$, $\APer_p$ and $C(\overline{\Rb})_p$ are 
quasi-banded as is shown there. Here is the 
remaining ingredient for the present paper:

\begin{lem}\label{LSIO}
The Cauchy singular integral operator $S$ belongs to $\pb{X}$, but is not
quasi-banded.
\end{lem} 
\begin{proof}
Let $u\in \Xb=L^{p}(\Rb)$ and $n>m$. Using the modulus inequality,
H\"older's inequality and $(q-1)\frac{p}{q}=1$ we obtain,
\begin{align*}
\|P_mS(I-P_n)u\|^p & \leq \| P_mS\chi_{[n,+\infty)}u\|^p+\|P_mS\chi_{(-\infty,-n]}u\|^p\\
 & =\frac{1}{\pi^p}\int_{-m}^m\left(\left|\int_{-\infty}^{-n}\frac{u(y)}{y-x}dy\right|^p
		+\left|\int_n^{+\infty}\frac{u(y)}{y-x}dy\right|^p\right)dx\\
 & \leq\frac{1}{\pi^p}\int_{-m}^m \|u\|^p 
		\left(\left(\int_{-\infty}^{-n}\frac{1}{|y-x|^q}dy\right)^{\frac{p}{q}}
			+\left(\int_n^{+\infty}\frac{1}{|y-x|^q}dy\right)^{\frac{p}{q}}\right)dx\\
 & \leq\frac{1}{\pi^p} \|u\|^p \int_{-m}^m\left(\frac{1}{(q-1)(n+x)^{q-1}}\right)^{\frac{p}{q}}
			+\left(\frac{1}{(q-1)(n-x)^{q-1}}\right)^{\frac{p}{q}}dx\\
 & =C\|u\|^p \int_{-m}^m\left(\frac{1}{n+x}+\frac{1}{n-x}\right)dx\\
 & =C\|u\|^p \cdot 2\ln\left(\frac{n+m}{n-m}\right),
\end{align*}
where $C=\pi^{-p}(q-1)^{-\frac{p}{q}}$. Since for any $m\in\Nb$
the last expression tends to $0$ as $n\rightarrow\infty$, we have
$\|P_mS(I-P_n)\|\rightarrow 0$ as $n\rightarrow\infty$. 

Now from the equality
$\|(I-P_{n})SP_{m}\|=\|P_m^*S^*(I-P_n)^*\|_{L^q}=\|P_mS(I-P_n)\|_{L^q}$
we obtain the dual formula: for every $m\in\Nb$, 
$\|(I-P_{n})SP_{m}\|\rightarrow 0$ as $n\rightarrow\infty$.
Finally, $S$ is not quasi-banded by \cite[Proposition 26]{MaSaSe2014}.
\end{proof}

\begin{cor}\label{CConvClass}
If $b\in[\PC,\SO,\BUC]_p$ then $W^0(b)\in\pb{\Xb}$. Moreover, 
if $b\in[\SO,\BUC]_p$ then $W^0(b)$ is quasi-banded.
\end{cor}
\begin{proof}
In the case $p=2$ we have already tackled all generators of $[\PC,\SO,\BUC]$
in the previous considerations.

If $p\neq 2$ then we again use an interpolation argument as in the proof of Corollary
\ref{CQBOComm} and combine it with \cite[Proposition 1.1.8]{RRS2004} which states 
that an operator $A$ belongs to $\pb{\Xb}$ if and only if 
\[\|P_mAQ_n\|+\|Q_nAP_m\|\to 0 \quad\text{as}\quad n\to\infty\quad\text{for every}\quad m\in\Nb.\]
Let $b\in[\PC,\SO,\BUC]\cap M^{<p>}$ and $r$ with
$|r-2|>|p-2|$ such that $b\in M^r$. Then $\|P_mW^0(b)Q_n\|+\|Q_nW^0(b)P_m\|\leq 2\|W^0(b)\|$
over $L^r(\Rb)$. Since $W^0(b)\in \pb{L^2(\Rb)}$ by the above, we have 
$\|P_mW^0(b)Q_n\|+\|Q_nW^0(b)P_m\|\to 0$ over $L^2(\Rb)$. By the Riesz-Thorin Interpolation 
Theorem we get convergence to zero also over $L^p(\Rb)$.

Finally, if $b\in[\SO,\BUC]\cap M^{<p>}$ then we use the same interpolation
argument to estimate the norms in \eqref{EDefQB}.
\end{proof}

\subparagraph{Summary}
We get the following picture for convolutions on $\Xb=L^p(\Rb)$: \\
All convolution operators with multipliers in the algebra $[\PC,\SO,\BUC]_p$ 
belong to $\pb{\Xb}$, but some of them are not in the class of quasi-banded operators,
in particular those with Fourier multipliers with jumps at finite points. 
In contrast, jumps at infinity are permitted and all functions in the algebra $[\SO,\BUC]_p$ 
yield quasi-banded operators.  Further, for all functions 
$b\in [\SO^\infty,\APer,\BUC]_p=[\BUC]_p$ the operators $W^0(b)$ are even banded-dominated. 
We point out that slowly oscillating multipliers with discontinuities at finite points 
do not generate band-dominated convolutions, but only quasi-banded operators.

\section{Setting up the algebraic framework}
The basic idea for the study of stability (and more general properties) of the
finite sections sequences is to embed these sequences $\Ab=\{A_n\}$ into a Banach 
algebra framework. In this framework, there is a family of homomorphisms which condense
$\Ab$ to single operators which we call snapshots. Intuitively speaking
we pursue a divide and conquer strategy: each one of these snapshots captures a certain 
easy part of the relevant properties of $\Ab$, and all of them together are sufficient
for the complete characterization of $\Ab$.

We start with defining this algebra framework on an abstract level and state all
relevant prerequisites as well as the main results. In Section \ref{SSAlgConv}
we  apply this tool to our problem and show that it covers the
classes of convolution type operators that we announced in the introduction.
At this stage, the general main theorem still contains a somewhat abstract and unhandy 
condition which is removed in Section \ref{SLocal} by localization. This then yields 
the main results of this paper, in particular Theorem \ref{TIntro}.
Finally, in Section \ref{SFlip} we propose a slightly modified implementation
of the framework which makes the snapshots a little more involved, but then 
additionally includes the flip operator, hence Hankel operators, as well.

\subsection{Algebras of structured and rich sequences}\label{SStruct}
Here we mainly follow \cite{MaSaSe2014, SeidelDiss, SeSi2012}. Let $\Xb$ be a 
Banach space. We are interested in the study of bounded sequences
$\Ab=\{A_n\}_{n\in\Nb}$ of bounded operators $A_n\in\lb{\Xb}$. By $\Fc$ we denote 
the set (actually the Banach algebra, equipped with the usual entry-wise defined 
linear structure and the norm $\|\Ab\|:=\sup_n\|A_n\|$) of all such bounded
sequences.

As announced, we further introduce a family (w.r.t. an index set $T$) of snapshots
$\Ws^t(\Ab)$, $t\in T$, which are bounded linear operators on certain Banach spaces
$\Xb^t$, respectively. This is done as follows:
Suppose that for each $t\in T$ there is a sequence of Banach algebra isomorphisms 
$E_n^t:\lb{\Xb^t}\to\lb{\Xb}$, $n\in \Nb$. Then the sequence 
$\{A_n^{(t)}\}:=\{(E_n^t)^{-1}(A_n)\}$ is a transformed copy of $\Ab$ of operators
now acting on $\Xb^t$.
Then, let $\Fc^T$ denote the set of all sequences $\Ab\in\Fc$ with the property that
for every $t\in T$ there exists an operator in $\lb{\Xb^t}$, which we denote by 
$\Ws^t(\Ab)$, such that
\[A_n^{(t)}=(E_n^t)^{-1}(A_n) \to \Ws^t(\Ab)\quad\text{as}\quad n\to\infty.\]
This convergence is to be understood as as one of the three following cases:
\begin{itemize}
\item Either all these limits are $*$-strong limits, which means that 
			$A_n^{(t)} \to \Ws^t(\Ab)$ and $(A_n^{(t)})^* \to (\Ws^t(\Ab))^*$ strongly, 
\item or on every space $\Xb^t$ there exists a uniform approximate identity $\Pc^t$ 
			which equips $\Xb^t$ with the $\Pc^t$-dichotomy, and the limits are 
			$\Pc^t$-strong limits, respectively. In that case $\Ws^t(\Ab)$ has to belong
			to $\lb{\Xb^t,\Pc^t}$,
\item or there is a partition of $T=T_*\cup T_\Pc$, and the limits for $t\in T_*$
			are supposed to be $*$-strong limits, whereas for all $t\in T_\Pc$ they
			are $\Pc^t$-strong limits as above, respectively.
\end{itemize}
We say that the sequences in $\Fc^T$ are $T$-structured, since they have a
certain asymptotic structure which is reflected in their snapshots.

\begin{figure}
	\centering
  \includegraphics[scale=1]{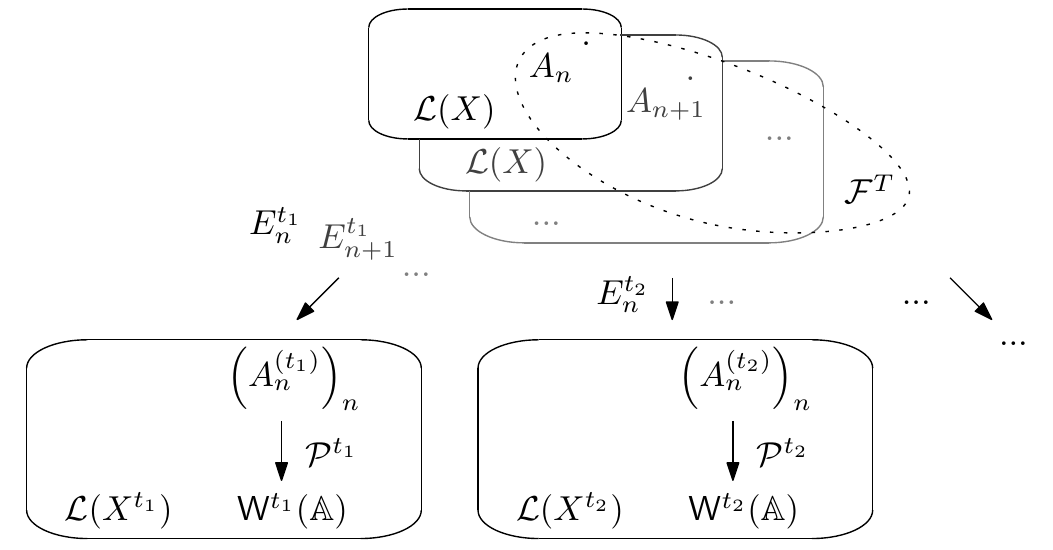}
  \caption{Taking snapshots: transformation and passing to limits}
  \label{fig:SAlgs}
\end{figure}

\subparagraph{Compact and Fredholm sequences}
Besides invertibility and stability, we also want to address a 
Fredholm property (kind of ``almost stability''), thus we have to introduce 
compact-like sequences as follows:

For $t\in T$ let $\Kc^t$ be the set of all compact or $\Pc^t$-compact operators, 
depending on the type of the convergence (i.e. depending on whether $t\in T_*$ or 
$t\in T_\Pc$). 
Then, let $\Jc^T$ be the smallest closed subspace of $\Fc$ which contains all sequences
$\{E_n^t(K)\}$, $K\in\Kc^t$, $t\in T$, and all sequences $\{G_n\}$, $\|G_n\|\to 0$
as $n\to\infty$.

\medskip
In order to make things fit together one has to impose two natural conditions:
\begin{itemize}
\item[(I)] For all $t\in T$ it holds that 
			$\sup\{\|E_n^t\|,\|(E_n^t)^{-1}\|:n\in\Nb\}<\infty$.
\item[(II)] For all $\tau,t\in T$ and every $K^t \in \Kc^t$ it holds that
	\begin{equation*}
		\Ws^\tau \{E^t_n (K^t)\}=
			\begin{cases}
				K^t& \text{if}\; t=\tau \\
				0  & \text{if}\; t\neq\tau
			\end{cases}.
	\end{equation*}
\end{itemize}
The first condition guarantees that the transformations are not too exotic, and the
second one is usually referred to as the separability condition and means that the directions 
from which one can look at a sequence are separated in the sense that the ($\Pc^t$-)compact 
operators $K^t$ and their liftings $\{E^t_n (K^t)\}$ which arise from one point of view $t\in T$ 
are invisible from every other direction $\tau$.

\medskip
In such a setting, $\Fc^T$ is a unital\footnote{Clearly, $\{I\}$ is its unit.} 
Banach algebra which includes $\Jc^T$ as a closed two-sided ideal. The mappings
$\Ws^t:\Fc^T\to\lb{\Xb^t}$, $t\in T$, which send each $\Ab$ to $\Ws^t(\Ab)$, respectively,
are Banach algebra homomorphisms.

A sequence $\Ab\in\Fc^T$ is said to be $\Jc^T$-Fredholm if $\Ab+\Jc^T$ is invertible
in the quotient algebra $\Fc^T/\Jc^T$. So this is ``invertibility modulo $\Jc^T$-compact
sequences''. 
Actually, if $\Ab$ is $\Jc^T$-Fredholm then all its snapshots are Fredholm, resp. 
$\Pc$-Fredholm. If all snapshots are Fredholm then $\Ab$ is said to be regular.

The collection of the main results for $T$-structured sequences reads as follows:

\begin{thm} \label{TTstructMain}
Let $\Ab=\{A_n\}\in\Fc^T$.
\begin{enumerate}
\item If $\Ab$ is a regular $\Jc^T$-Fredholm sequence then, for sufficiently large $n$, 
			the operators $A_n$ are Fredholm and all snapshots $\Ws^t(\Ab)$ are Fredholm. Moreover,
			their approximation numbers from the right have the $\ah (\Ab)$-splitting 
			property, that is
			\[\lim_{n\to\infty} s_{\ah (\Ab)}^r(A_n) = 0 \quad\text{and}\quad
				\liminf_{n\to\infty} s_{\ah (\Ab) + 1}^r(A_n) > 0,\]
			with
			\[\alpha(\Ab)=\sum_{t\in T}\dim\ker \Ws^t(\Ab)<\infty.\]
			Analogously, the approximation numbers from the left have the 
			$\bh (\Ab)$-splitting property, where $\beta(\Ab)=\sum_{t\in T}\dim\coker \Ws^t(\Ab)$
			is finite. Furthermore,
			\[\lim_{n\to\infty} \ind A_n = \sum_{t\in T} \ind \Ws^t(\Ab).\]
\item If one snapshot is not Fredholm then, for all $k\in\Nb$, 
			$\min\{s_k^r(A_n), s_k^l(A_n)\}\to0$ as $n\to\infty$.
\item $\Ab$ is stable if and only if $\Ab$ is $\Jc^T$-Fredholm and all snapshots 
			are invertible.
\end{enumerate}
\end{thm}
So, roughly speaking, in the Fredholm case, the Fredholm properties of   ${A_n}$ are 
captured by the snapshots, in particular the stability is connected with the invertibility 
of all $\Ws^t(\Ab)$. Otherwise, if one snapshot is not Fredholm then also the
$A_n$ cannot be close to Fredholmness.

Of course, we have to recall the definitions of the approximation numbers which are 
used here to gather the Fredholm properties of the $A_n$ from \cite{SeSi2012}
\begin{align*}
s_k^r(A)&:=\inf\{\|A-T\|:T\in\lb{\Xb}, \dim\ker T \geq k\},\\
s_k^l(A)&:=\inf\{\|A-T\|:T\in\lb{\Xb}, \dim\coker T \geq k\}.
\end{align*}
In a sense, these numbers measure the degree of (or distance to) injectivity and
surjectivity, respectively. Also notice that in case of $\Xb$ being a 
Hilbert space, these approximation numbers coincide with the lower singular 
values of $A$ or $A^*$, respectively.

\subparagraph{Rich sequences}
Actually, for our purposes in Section \ref{SSAlgConv} this framework of $T$-structured 
sequences is still too restrictive, and what we need is a tool for sequences $\Ab$ 
which may not be $T$-structured (as a full sequence), but which have at least sufficiently 
many subsequences $\Ab_g=\{A_{g_n}\}$ (where $g:\Nb\to\Nb$ is strictly increasing)
being $T$-structured.

To be more precise, we firstly point out, that one can of course fix such a strictly 
increasing sequence $g=(g_n)$ and consider the above machinery for subsequences 
$\Ab_g=\{A_{g_n}\}$ instead. In an analogous way to $\Fc$ and $\Fc^T_g$ 
we define $\Fc_g$  as the set of all bounded sequences
$\Ab_g$ and $\Fc^T_g$ as the set of all $\Ab_g\in\Fc_g$ for which all limits
$\Ws^t(\Ab_g):=\lim_{n\to\infty}(E_{g_n}^t)^{-1}(A_{g_n})$ exist. Clearly, 
the analogous conditions (I) and (II) hold. Define, further, the set $\Jc^T_g$ 
of ``compact subsequences'' and the notion of $\Jc^T_g$-Fredholmness.  
Of course, Theorem \ref{TTstructMain} can be translated 
to this subsequence framework.

\medskip
Here is now the next step:
A sequence $\Ab\in\Fc$ is rich if every subsequence of $\Ab$ has a $T$-structured
subsequence $\Ab_g=\{A_{g_n}\}$, i.e. $\Ab_g\in\Fc^T_g$. Denote the set of all rich
sequences by $\Rc^T$. 
Furthermore, for $\Ab\in\Rc^T$ we denote by $\Hb_\Ab$ the set of all strictly increasing
$g:\Nb\to\Nb$ for which $\Ab_g\in\Fc^T_g$. Also, all snapshots of $T$-structured 
subsequences of $\Ab$ are referred to as snapshots of $\Ab$.	

\begin{thm}\label{TRich}
It holds that
\begin{enumerate}
\item $\Fc\supset\Rc^T\supset\Fc^T$ are Banach algebras and $\Rc^T$ is inverse closed 
			in $\Fc$.
\item If for $\Ab=\{A_n\}\in\Rc^T$ every $T$-structured subsequence has a regularly
			$\Jc^T$-Fredholm subsequence then $\Ab$ has finite $\alpha$- and $\beta$-number 
			$\alpha(\Ab)$, $\beta(\Ab)$, i.e.
			\begin{align*}
			\liminf_{n\to\infty}& s^r_{\alpha(\Ab)}(A_n)=0, 
				&\liminf_{n\to\infty}& s^l_{\beta(\Ab)}(A_n)=0,\\
			\liminf_{n\to\infty}& s^r_{\alpha(\Ab)+1}(A_n)>0, 
				&\liminf_{n\to\infty}& s^l_{\beta(\Ab)+1}(A_n)>0,
			\end{align*}
			where
			\begin{align*}
			\alpha(\Ab)=\max_{h\in\Hb_\Ab}\sum_{t\in T}\dim\ker \Ws^t(\Ab_h),\quad
			\beta(\Ab)=\max_{h\in\Hb_\Ab}\sum_{t\in T}\dim\coker \Ws^t(\Ab_h).
			\end{align*} 
\item If one snapshot of $\Ab\in\Rc^T$ is not Fredholm then $\Ab$ cannot have both, 
			finite $\alpha$- and $\beta$-number, i.e. for every $k\in\Nb$
			\[\liminf_{n\to\infty} \min\{s^r_{k}(A_n), s^l_{k}(A_n)\} =0.\]
\item For $\Ab\in\Rc^T$ the following are equivalent
			\begin{enumerate}
			\item $\Ab$ is stable.
			\item Every $T$-structured subsequence of $\Ab$ is stable.
			\item Every $T$-structured subsequence of $\Ab$ has a stable subsequence.
			\item Every $T$-structured subsequence of $\Ab$ is $\Jc^T$-Fredholm 
						and all snapshots of $\Ab$ are invertible.
			\item Every $T$-structured subsequence of $\Ab$ has a $\Jc^T$-Fredholm 
						subsequence and all snapshots of $\Ab$ are invertible.
			\item $\alpha(\Ab)$ and $\beta(\Ab)$ exist and are both equal to zero. 
			\end{enumerate}
\end{enumerate}
\end{thm} 

\subparagraph{Some comments about the proofs}
A thorough exposition of this model with full proofs and applications can be 
found in \cite[Section 2]{SeidelDiss}, even in a slightly more general form. 
However, the results on $T$-structured sequences are also introduced and proved 
in \cite[Section 2]{SeSi2012}. Theorem \ref{TRich} (except 3.) can be proved  
as was done in \cite[Theorem 8]{MaSaSe2014}. Its 3rd assertion immediately follows
from \cite[Theorem 2.26 and Corollary 2.27]{SeSi2012}. 

\subsection{The sequence algebra framework for convolution type operators}\label{SSAlgConv}
Now we come back to our initial problem, the finite section method for convolution
type operators on $\Xb:=L^p(\Rb)$, and we apply the general theory of the previous section to
this concrete situation. This actually means that we are aiming for a framework
which includes the algebra
\footnote{Recall that $\APer\subset\BUC$.}
\begin{equation}\label{EAc}\begin{split}
\Ac:=\alge&\left\{ \{aI\},\,\{W^0(b)\},\,\{K\},\,\{P_n\}:\right. \\
 & \;\;\left. a\in[\PC,\SO,L_0^{\infty}],\; b\in[\PC,\SO,\BUC]_p,
 \;K\in\pc{X}\right\}.
\end{split}
\end{equation}
For this we set $T:=T_\Pc\cup T_*$ with $T_\Pc:=\{-,c,+\}$ and $T_*:=\Rb$, and further
$\Xb^t:=L^p(\Rb)$ for every $t\in T$, where $\Pc^t:=\Pc=(P_n)=(\chi_{[-n,n]}I)$
if $t\in T_\Pc$. Then,  $\Fc^T$  is  the set of all $\Ab=\{A_n\}\in\Fc$ 
for which the $\Pc$-strong limits
\begin{align*}
\Ws^c(\Ab):=\plimn A_n,\quad \Ws^-(\Ab):=\plimn V_nA_nV_{-n},\quad
\Ws^+(\Ab):=\plimn V_{-n}A_nV_n
\end{align*}
and, for every $t\in T_*$, the $*$-strong limits
\[\Hs^t(\Ab):=\slim Z_n^{-1} U_t A_n U_{-t} Z_n\]
exist, where we recall the homomorphisms
\begin{align*}
Z_n: (Z_nu)(x)= n^{-1/p} u(x/n),\quad
U_t: (U_t u)(x)= e^{\ii tx} u(x) \quad\text{on}\quad L^p(\Rb)
\end{align*}
from \cite{KMS2010}.
Clearly, we have to ensure that 
\begin{lem}\label{LSepC}
The Conditions {\upshape (I)} and {\upshape (II)} are satisfied. 
\end{lem}
\begin{proof}
Since all $V_n$, $Z_n$ and $U_t$ are isometric isomorphisms, the first condition
is obvious. For the second one we start with $K:=P_m=\chi_{[-m,m]}I$, and recall 
from \cite{MaSaSe2014} that for the sequence $\{E_n^c(K)\}=\{K\}$
the $\Ws^\pm$-snapshots are zero. Further, all $\Hs^t$-snapshots are zero by
\cite[Lemma 7.4.(b)]{KMS2010}. For $\{K\}$ with arbitrary $K\in\pc{X}$ 
just use the observation $K=P_mKP_m+K(I-P_m)+(I-P_m)KP_m$. Then, given any 
prescribed $\epsilon>0$, choose and fix $m$ sufficiently large such that 
the second and third summand get less than $\epsilon$ and observe the desired 
convergence for the first summand. So, $\Ws^\pm(\{K\})=0$ and $\Hs^t(\{K\})=0$,
$t\in\Rb$, easily follow.  
Liftings $\{E_n^\pm(K)\}=\{V_{\pm n} K V_{\mp n}\}$ are treated similarly.

Now, for $t\in\Rb$ consider $\{J_n\}=\{U_{-t} Z_n K Z_n^{-1} U_t\}$
with  $K$ a compact operator. The limits $\Hs^s\{J_n\}=0$ for $s\neq t$ were shown 
in \cite[Proposition 4.3]{DPSS2014}.
Next,
\[\|P_m U_{-t} Z_n K Z_n^{-1} U_t\| \leq
\|U_{-t} Z_n Z_n^{-1} U_t P_m U_{-t} Z_n K \| 
\leq \|Z_n^{-1} U_t P_m U_{-t} Z_n K\| \to 0\]
as $n\to\infty$, for every $m$, as discussed above. By duality, $\Ws^c\{J_n\}=0$ 
follows. The snapshots $\Ws^\pm\{J_n\}=0$ are proved analogously and the rest
is straightforward.
\end{proof}

Thus, we have the above Theorems \ref{TTstructMain} and \ref{TRich} available, and 
the rest of this section is devoted to the proof that this applies to all sequences
in $\Ac$:
\begin{prop}
$\Ac\subset\Rc^T$.
\end{prop}
In fact, all we have to do is to check that the generators of $\Ac$ are rich. 
For $\{K\}$ it is already clear by Lemma \ref{LSepC}, and the remaining cases 
are treated in the subsequent lemmas, where we also compute their snapshots.
Define $\chi_-:=\chi_{(-\infty,0]}$ and $\chi_+:=\chi_{[0,+\infty)}$.

\begin{lem}\label{LRMultOps}
For every $a\in L^\infty(\Rb)$ it holds that $\Ws^c\{aI\}=aI$.\\
For $a\in L^\infty_0\cup\PC\cup\bigcup_{\lambda\in\Rb}\SO^\lambda$ we have 
$\{aI\}\in\Fc^T$ where, for all $t\in\Rb$,
\begin{align*}
\Ws^-\{aI\}=a(-\infty)I,\quad \Ws^+\{aI\}=a(+\infty)I,
\quad \Hs^t\{aI\}=(a(-\infty)\chi_-+a(+\infty)\chi_+)I.
\end{align*}
For $a\in\SO^\infty$ it holds that $\{aI\}\in\Rc^T$, where all snapshots 
$\Ws^t(\{aI\}_g)$, $g\in\Hb_{\{aI\}}$, $t\in T\setminus\{c\}$, are multiples of
the identity.
\end{lem} 
\begin{proof}
The $\Ws$-snapshots for all functions under consideration are almost obvious or 
have already been discussed in \cite{MaSaSe2014}.

Since $a\in L^\infty_0$ yields $aI\in\Kc^c=\pc{\Xb}$, hence $\{aI\}\in\Jc^T$, we
see that all $\Hs$-snapshots are zero by Lemma \ref{LSepC}.

Functions $a\in \PC\cup\bigcup_{\lambda\in\Rb}\SO^\lambda$ can be decomposed  
$a=a(-\infty)\chi_-+a(+\infty)\chi_++a_0$ with $a_0\in L^\infty_0$, and 
$\chi_\pm$ are tackled in \cite[Proposition 7.4.b]{KMS2010}.

Finally, let $a\in \SO^\infty$ and $g:\Nb\to\Nb$ be strictly increasing. 
From \cite[Prop. 3.39]{Li2006} or \cite{MaSaSe2014} it is already known that 
there is a subsequence $h\subset g$
such that the $\Ws$-snapshots of $\{aI\}_h$ exist. It remains to show that for 
a certain subsequence $l$ of $h$ one also has the existence of all $\Hs^t$, $t\in\Rb$. 
For this we consider
\[(Z_{h_n}^{-1} U_t a U_{-t} Z_{h_n} u)(x)= (Z_{h_n}^{-1} a Z_{h_n} u)(x)=a(h_nx) u(x).\]
The set of all continuous functions $u$ with compact support $U_u$ such that $0\notin U_u$ 
is dense in $L^p(\Rb)$, thus it suffices to consider such $u$. Fix one point $x_0\in U_u$. 
Then, by a Bolzano-Weierstrass argument, there is a subsequence $l$ of $h$ such 
that $a(l_nx_0)$ converges, lets say to $a_0$. Since $a\in \SO^\infty$ 
the oscillation $\osc(a(l_n\cdot), U_u)$ tends to zero, hence the sequence 
of the functions $x\mapsto a(l_nx)u(x)$ converges uniformly to the 
function $a_0u(x)$, which gives the claim.
\end{proof}

\begin{lem}\label{LW0bComm}
For every Fourier multiplier $b$ it holds that
\[\Ws^c\{W^0(b)\}=\Ws^-\{W^0(b)\}=\Ws^+\{W^0(b)\}=W^0(b).\]
If $b\in[\BUC,\PC]_p$ then $\{W^0(b)\}\in\Fc^T$ with
\[\Hs^t\{W^0(b)\}=W^0(b(t-0)\chi_-+b(t+0)\chi_+) \quad (t\in\Rb).\]
If  $b\in[\BUC,\PC,\SO]_p$ then $\{W^0(b)\}\in\Rc^T$ and the $\Hs^t$-snapshots
(of $T$-structured subsequences) are still of the form 
$W^0(c_-\chi_-+c_+\chi_+)$ with constants $c_-, c_+$.
\end{lem} 
\begin{proof}
The $\Ws$-snapshots are clear, since convolution operators are shift invariant.

For the $\Hs$-snapshots we first assume that $p=2$. Then it suffices to consider 
the generators $W^0(b)$ separately with $b$ either uniformly continuous over $\Rb$, 
or $b=\chi_{[\lambda,\infty)}$, or $b\in\SO^\lambda$ for $\lambda\in \Rb$. 
The piecewise continuous and the slowly oscillating $b$ are already studied in 
\cite[Proposition 7.4.c]{KMS2010}.
so, let $b\in\BUC$ and $\epsilon>0$. Then there exists a $\delta>0$ such that
$\sup\{\osc(b,[t-\delta,t+\delta]):t\in\Rb\}<\epsilon$. Let $\varphi^t$
be the continuous piecewise linear spline which is $1$ at the point $t$ and $0$
outside the interval $[t-\delta,t+\delta]$, resp. Then
\[W^0(b)=W^0(b\varphi^t)+W^0(b(1-\varphi^t))
=b(t)W^0(\varphi^t)+W^0((b-b(t))\varphi^t)+W^0(b(1-\varphi^t)),\]
where the first term has the $\Hs^t$-snapshot $b(t)I$, the third one has 
$\Hs^t$-snapshot zero, and the second one has norm less then $\epsilon$.
Since $\epsilon$ and $t$ were arbitrarily chosen, we find that $\Hs^t(W^0(b))$ 
must be $b(t)I$ for every $t\in\Rb$.

Now let $p\neq2$ and fix $b\in[\PC,\SO,\BUC]_p$ and $g:\Nb\to\Nb$. 
We can already assume that $b\in M^r$ for an $r$ with $|2-p|<|2-r|$. Indeed,
this is clear since, by the definition of the algebra $[\cdot]_p$, these
particular functions are dense, and moreover the 
$\Hs$-homomorphisms are bounded (even uniformly w.r.t. $t\in\Rb$).
The sequence $\{W^0(b)\}_g$, considered as a sequence over 
$L^2(\Rb)$ has a subsequence $\{W^0(b)\}_h\in\Fc_h^T$. We want to show that then
all $\Hs^t\{W^0(b)\}_h$, $t\in\Rb$, over $L^p(\Rb)$ exist as well. By symmetry 
reasons it suffices to consider $t=0$. Note that the snapshot $\Hs^0\{W^0(b)\}_h$
over $L^2(\Rb)$ is of the form $W^0(b_0)$ where $b_0=c_-\chi_-+c_+\chi_+$ with
constants $c_-,c_+$.

Let $d$ denote the continuous piecewise linear spline which is $1$ over $[-1,1]$
and $0$ for $|x|\geq 2$.  
Then, as shown in \cite[Proposition 7.4.c]{KMS2010}, $Z_m^{-1}W^0(d)Z_m$ are of 
the form $W^0(d^1_m)$ with $d^1_m(x)=d(x/m)$ and tend $*$-strongly to the identity, 
whereas $Z_mW^0(d)Z_m^{-1}$ are of the form $W^0(d^2_m)$ with $d^2_m(x)=d(mx)$ and 
tend $*$-strongly to zero. Thus, defining $d^3_m:=d^1_m(1-d^2_m)$ we obtain a sequence 
of operators $W^0(d_m^3)$ which tends $*$-strongly to the identity (see Figure \ref{fig:d3n-2}). 
For $u\in L^p(\Rb)$
consider the decomposition
\begin{align*}
&\|(Z_{h_n}^{-1} W^0(b) Z_{h_n}-W^0(b_0)) u\|
=\|Z_{h_n}^{-1} W^0(b-b_0) Z_{h_n} u\|  \\
&\quad\leq \|W^0(b-b_0)\|\|W^0(1-d_m^3)u\|  + \|Z_{h_n}^{-1} W^0(b-b_0) Z_{h_n}W^0(d_m^3) u\|\\
&\quad\leq  \|W^0(b-b_0)\|\|W^0(1-d_m^3)u\|  + \|W^0((b_{h_{n}}-b_0)d_m^3)\|\|u\|
\end{align*}
with $b_{h_{n}}(x) := b(x/h_{n})$ and let $\epsilon>0$. 
\begin{figure}
	\centering
  \includegraphics[scale=1.2]{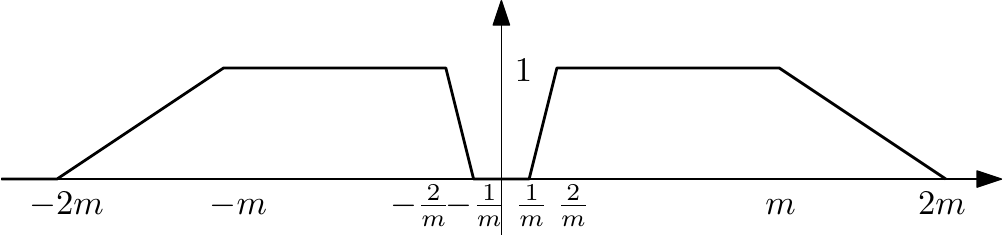}
  \caption{the functions $d_m^3$}
  \label{fig:d3n-2}
\end{figure}
For sufficiently large $m$ the 1st summand is less than 
$\epsilon\|u\|$ since the norms $\|W^0(1-d_m^3)u\|$ tend to zero as $m\to\infty$. 
Fix such an $m$. Then, for the 2nd summand notice that all the functions 
$(b_{h_n}-b_0)d_m^3$, $n\in\Nb$, are continuous at $0$, have their supports 
contained in the compact support of $d_m^3$ and converge pointwise to zero as $n\to\infty$. 
We claim that they even converge uniformly. If $b$ has only finitely many 
points of discontinuity, then this is obvious since, for sufficiently large $n$,  
the points of discontinuity of $b_{h_n}-b_0$ are outside the support of $d_m^3$,
hence these $(b_{h_n}-b_0)d_m^3$ are uniformly continuous and the claim follows as in 
the proof of Lemma \ref{LRMultOps}. For general $b$ we just apply the density 
of the set of all functions with finitely many discontinuities and an 
approximation argument. Thus, 
$\|W^0((b_{h_n}-b_0)d_m^3)\|_{\lb{L^2(\Rb)}}\leq \|(b_{h_n}-b_0)d_m^3\|_\infty\to 0$
as $n\to\infty$. Moreover 
$\|W^0((b_{h_n}-b_0)d_m^3)\|_{\lb{L^r(\Rb)}}\leq \|W^0(b-b_0)\|_{\lb{L^r(\Rb)}}$
are uniformly bounded. Using Riesz-Thorin interpolation we find that 
$\|W^0((b_{h_n}-b_0)d_m^3)\|_{\lb{L^p(\Rb)}}$, hence also the second summands above,
tend to zero. Since $\epsilon$ was arbitrarily fixed, we get that 
$Z_{h_n}^{-1} W^0(b) Z_{h_n} u$ tends to $W^0(b_0)u$. 
Duality gives the rest. 
\end{proof}

We point out that $V_\lambda=W^0(b)$ with $b(x)=e^{\ii\lambda x}$ 
\cite[Sections 3.1, 19.1]{BKSp2002}, hence
\begin{cor}
All shift operators $V_\lambda$ induce sequences $\{V_\lambda\}\in\Ac\subset\Rc^T$.
\end{cor}

Finally, it is immediate that
\begin{lem}For the sequence $\{P_n\}$ and for every $t\in\Rb$ it holds
\[\Ws^c\{P_n\}=I,\quad \Ws^-\{P_n\}=\chi_+I, \quad\Ws^+\{P_n\}=\chi_-I, 
\quad \Hs^t\{P_n\}=P_1.\]
\end{lem} 

\subsection{Finite sections of convolution type operators}\label{SLocal}

Theorem \ref{TRich} gives us stability conditions and a characterization
of the asymptotic behavior of the approximation numbers for a wide
class of operator sequences: the $T$-rich sequences. To know if the
conditions are satisfied, we need to know whether or when for a given
sequence $\Ab$ every $T$-structured subsequence has a subsequence
$\Ab_{h}$ which is regularly $\Jc^{T}$-Fredholm, that means a subsequence
$\Ab_{h}$ such that $\Ab_{h}+\Jc_{h}^{T}$ is invertible in the quotient
algebra $\Fc_{h}^{T}/\Jc_{h}^{T}$ and all operators $\Ws^{t}(\Ab_{h})$
and $\Hs^{t}(\Ab_{h})$ are Fredholm.

We will show, in particular, for finite sections of operators in a
large class of convolution type operators, that the  $\Pc$-Fredholmness of
the three snapshots $\Ws^{t}(\Ab_{h})$ imply the invertibility in $\Fc_{h}^{T}/\Jc_{h}^{T}$.

In order to apply Allan-Douglas principle, we start by defining, for
each continuous function $\varphi\in C(\dot{\Rb})$, the sequence of
expanded functions given by 
\[\varphi_{n}:\Rb\to\Rb,\quad\varphi_{n}(x):=\varphi(x/n).\]

Define further, $\Cc:=\{\{\varphi_{n}I\}:\varphi\in C(\dot{\Rb})\}$
and note that an alternative representation is given by
\begin{equation}
\{\varphi_{n}I\}=\{Z_{n}\varphi Z_{n}^{-1}\}.\label{eq:expanded}
\end{equation}

\begin{prop}
$\Cc$ is a closed commutative subalgebra of $\Fc^{T}\subset\Rc^{T}$
with 
\[\Ws^{\pm}\{\varphi_{n}I\}=\varphi(\pm1)I, \quad
\Ws^{c}\{\varphi_{n}I\}=\varphi(0)I\quad\text{and}\quad
\Hs^{t}\{\varphi_{n}I\}=\varphi I \text{ for all $t\in\Rb$};\] 
and is isometrically isomorphic to $C(\dot{\Rb})$.
\end{prop}

\begin{proof}
Note that $\Ws^{+}\{\varphi_{n}I\}=\mathcal{P}\mbox{-}\lim_{n\rightarrow\infty}\phi_{n}I$,
with $\phi_{n}(x):=\varphi(\frac{x+n}{n})$. Since $\phi_{n}$ converges
pointwise to $\varphi(1)$ and $\varphi\in C(\dot{\Rb})$, then $\varphi_{n}I$
converges $\mathcal{P}$-strongly to $\varphi(1)I$. The same arguments
apply to the other W-snapshots. For the H-snapshots, the assertion
is straightforward taking into account \eqref{eq:expanded}. Thus,
$\Cc$ is a closed commutative subalgebra of $\Fc^{T}$ and it is
also clear that the map $\{\varphi_{n}\}\mapsto\varphi$ is an isometric
isomorphism.
\end{proof}

For sequences $\mathbb{A}$ and $\mathbb{B}$, let $\left[\mathbb{A},\mathbb{B}\right]$
denote the commutator $\mathbb{A}\mathbb{B}-\mathbb{B}\mathbb{A}$.

\begin{prop}\label{PComm1}
For every $\varphi\in C(\dot{\Rb})$,
\begin{itemize}
\item[(a)] $\Vert[\varphi_{n}I, W^{0}(b)]\Vert\rightarrow0$ for $b\in[\SO,\BUC]_{p}$ 
\item[(b)] $[\{\varphi_{n}I\},\{W^{0}(\chi_{[t,\infty)})\}]=\{U_{-t}Z_{n}KZ_{n}^{-1}U_{t}\}$,
with $K$ a compact operator.
\end{itemize}
\end{prop}
\begin{proof}
Assertion (a) is a consequence of Corollary \ref{CConvClass} and Theorem \ref{TQBCom}. 
For the symbols $b=\chi_{t}=\chi_{[t,\infty)}$, $t\in\mathbb{R}$, from the relations
$W^{0}(\chi_{t})=U_{-t}W^{0}(\chi_{+})U_{t}$ and 
$Z_{n}^{-1}W^{0}(\chi_{+})Z_{n}=W^{0}(\chi_{+})$
for every $t\in\mathbb{R}$, it follows that 
\begin{eqnarray*}
\varphi_{n}W^{0}(\chi_{t})-W^{0}(\chi_{t})\varphi_{n}I & = & \varphi_{n}U_{-t}W^{0}(\chi_{+})U_{t}-U_{-t}W^{0}(\chi_{+})U_{t}\varphi_{n}I\\
 & = & U_{-t}Z_{n}\varphi Z_{n}^{-1}W^{0}(\chi_{+})U_{t}-U_{-t}W^{0}(\chi_{+})Z_{n}\varphi Z_{n}^{-1}U_{t}\\
 & = & U_{-t}Z_{n}[\varphi I,W^{0}(\chi_{+})]Z_{n}^{-1}U_{t},
\end{eqnarray*}
and since $[\varphi I,W^{0}(\chi_{+})]$ is a compact operator (see \cite[Prop 5.3.1]{RSS2010book})
we obtain assertion (b). 
\end{proof}

We now introduce a subalgebra $\Ac_1$ of $\Ac$ by
\begin{equation}
\begin{split}
\Ac_1:=\alge&\left\{ \{aI\},\,\{W^0(b)\},\,\{K\},\,\{P_n\}:\right. \\
 & \;\;\left.a\in[\PC,\SO,L_0^{\infty}],\; b\in[\PC^\lambda,\SO,\BUC]_p,\;\lambda\in\Rb,
 \;K\in\pc{X}\right\}
\end{split}
\end{equation}
and mention that it is not known if $\Ac_1$ is really a proper subalgebra of $\Ac$, since it is
unknown whether there are multipliers in $\PC_p$ which do not belong to the smallest closed
algebra which includes all $[\PC^\lambda,\SO,\BUC]_p$, $\lambda\in\Rb$.
Anyway, this algebra covers, unites and extends all previously considered cases in the literature.

\begin{cor}\label{CComm1}
Let $\mathbb{A}\in\Ac_1$. Then $\left[\mathbb{A},\mathbb{C}\right]\in\Jc^{T}$
for every $\mathbb{C}\in\Cc$.
\end{cor}
\begin{proof}
It suffices to check the assertion for the generators of $\Ac_1$. Clearly, constant
sequences $\{aI\}$ of operators of multiplication, as well as $\{P_{n}\}$
commute with $\{\varphi_{n}I\}$, and $\{K\}\in\Jc^T$ by definition. The generators
$\{W^0(b)\}$ are covered by the previous proposition, since every $b\in[\PC^\lambda,\SO,\BUC]_p$
can be decomposed as $b=\alpha \chi_\lambda + c$ with $\alpha\in\Cb$ and $c\in[\SO,\BUC]_p$, hence
\[[\{\varphi_n I\},\{W^{0}(b)\}]=\alpha\{U_{-\lambda}Z_{n}[\varphi I,W^{0}(\chi_{+})]Z_{n}^{-1}U_{\lambda}\}
+[\{\varphi_n I\},\{W^{0}(c)\}],\]
where the latter summand tends to $0$ as $n\to\infty$ and the first one is in $\Jc^T$.
\end{proof}

Let $h$ be an increasing sequence of natural numbers, $\Cc_{h}:=\{\mathbb{C}_{h}:\,\mathbb{C}\in\mathcal{C}\}$
and let $\Lc_{h}^{T}\subset\Fc_{h}^{T}$ be the set of \textit{sequences of
local type}, i.e. the set of all sequences $\Ab_h\in\Fc_{h}^{T}$ with
$\Ab_h\Cb_h-\Cb_h\Ab_h\in\Jc_{h}^{T}$ for all $\mathbb{C}\in\Cc$.

Then \cite[Lemma 6.2]{KMS2010} translates to the present setting and
provides that $\Lc_{h}^{T}$ is a closed and inverse closed subalgebra
of $\Fc_{h}^{T}$ containing $\Jc_{h}^{T}$ as a closed ideal. Moreover,
$\Lc_{h}^{T}/\Jc_{h}^{T}$ is inverse closed in $\Fc_{h}^{T}/\Jc_{h}^{T}$.
The algebra $\Cc_{h}^{\mathcal{J}}:=(\Cc_{h}+\Jc_{h}^{T})/\Jc_{h}^{T}$
is a closed central subalgebra of $\Lc_{h}^{T}/\Jc_{h}^{T}$ which
is isomorphic to $\Cc\cong C(\dot{\Rb})$. To see the latter, just
assume that the homomorphism $C(\dot{\Rb})\rightarrow\Lc_{h}^{T}/\Jc_{h}^{T}$
given by $\varphi\mapsto\{\varphi_{h_n}I\}+\Jc_{h}^{T}$ is not injective.
Then $\{\varphi_{h_n}I\}+\Jc_{h}^{T}=\Jc_{h}^{T}$ for some $\varphi\not\equiv 0$,
i.e. $\{\varphi_{h_n}I\}\in\Jc_{h}^{T}$. But then the snapshot $\Hs^{0}\{\varphi_{h_n}I\}=\varphi I$
is a compact operator, hence $\varphi\equiv0$, which is a contradiction. 
Consequently, the maximal ideal space of $\Cc_{h}^{\mathcal{J}}$ is
homeomorphic to $\dot{\Rb}$. Denote by $\Phi_{s}$ the canonical
homomorphism from $\Lc_{h}^{T}/\Jc_{h}^{T}$ onto the local algebra
at $s\in\dot{\Rb}$. Then, using Allan's local principle (see e.g. \cite[Theorem 2.2.2]{RSS2010book}),
a sequence $\Ab_h\in\Lc_{h}^{T}$ is $\Jc_{h}^{T}$-Fredholm if and only if $\Phi_{s}(\Ab_h+\Jc_{h}^{T})$
is invertible for every $s\in\dot{\Rb}$. We want to identify these local cosets, and 
for this goal we consider a subalgebra of $\mathcal{A}_1$ which still
contains the finite sections we are interested in.

\begin{defn}
Let $\Fc_{\Ac_1}$ be the smallest closed subalgebra of $\Rc^{T}$ which
contains all finite section sequences $\{P_{n}A_{n}P_{n}+(I-P_{n})\}$,
with $\{A_{n}\}\in\Ac_1$. 
Clearly, $\Fc_{\Ac_1}\subset\mathcal{A}_1$ and it is not hard to check
that each of its elements $\Ab=\{A_{n}\}$ is of the form
\[\{A_{n}\}=\{P_{n}A_{n}P_{n}+\mu_\Ab(I-P_{n})\}\]
with a certain complex number $\mu_\Ab$. 
\end{defn}

\begin{prop}\label{localization}
Let $\Ab\in\Fc_{\Ac_1}$ and $g:\Nb\to\Nb$ be strictly increasing. 
Then there exists a subsequence $h\subset g$ such that $\Ab_h\in\Lc_{h}^{T}$ and
\begin{itemize}
\item[(a)] $\Phi_{s}(\mathbb{A}_{h}+\Jc_{h}^{T})=\Phi_{s}(\{\Ws^{c}(\Ab_{h})\}+\Jc_{h}^{T})$ if $|s|<1$
\item[(b)] $\Phi_{\pm1}(\mathbb{A}_{h}+\Jc_{h}^{T})=\Phi_{\pm1}(\{V_{\pm h_{n}}\Ws^{\pm1}(\Ab_{h})V_{\mp h_{n}}\}+\Jc_{h}^{T})$
\item[(c)] $\Phi_{s}(\mathbb{A}_{h}+\Jc_{h}^{T})=\mu_\Ab\Phi_{s}(\{I\}+\Jc_{h}^{T})$ if $|s|>1$.
\end{itemize}
If $\Ws^{c}(\Ab_{h})$ (or $\Ws^{\pm1}(\Ab_{h})$) is $\Pc$-Fredholm and $B$ (or $B_{\pm}$, resp.) 
is a $\Pc$-regularizer, then the sequence $\{B\}_h$ (or $\{V_{\pm h_{n}}B_{\pm}V_{\mp h_{n}}\}_h$, resp.)
belongs to $\Lc_{h}^{T}$ as well. 
\end{prop}

\begin{proof}
Since $\Ab$ is rich we can always easily pass to $T$-structured subsequences $\Ab_h$, and then
Corollary \ref{CComm1}  yields $\Ab_h\in\Lc_{h}^{T}$. Moreover, it is easily seen that
$\{\Ws^{c}(\Ab_{h})\}$ and  $\{V_{\pm h_{n}}\Ws^{\pm1}(\Ab_{h})V_{\mp h_{n}}\}$  belong 
to $\Fc^T_h$ as well.

Let $|s|<1$ and choose a continuous function $\varphi$ such that
$\supp\varphi\subset[-1,1]$ and $\varphi(s)=1$. 
Then $\Phi_{s}(\{\varphi_{h_{n}}I\}+\Jc_{h}^{T})=\Phi_{s}(\{I\}+\Jc_{h}^{T})$
and $\Phi_{s}(\{\varphi_{h_{n}}(I-P_{h_{n}})\}+\Jc_{h}^{T})=0$ which
implies $\Phi_{s}(\{P_{h_{n}}\}+\Jc_{h}^{T})=\Phi_{s}(\{I\}+\Jc_{h}^{T})$.
Thus, 
\[\Phi_{s}(\Ab_{h}+\Jc_{h}^{T})=\Phi_{s}(\{\Ws^{c}(\Ab_{h})\}+\Jc_{h}^{T}).\]
Actually, this already holds for the full sequence $\Ab$. 

For the local algebra at $s=1$, we will check the assertion even for all sequences in 
$\Ac_1$, and we start with the generators. It holds
that $\Phi_{1}(\{P_{h_{n}}\}+\Jc_{h}^{T})=\Phi_{1}(\{V_{h_{n}}\chi_{-}V_{-h_{n}}\}+\Jc_{h}^{T})$
and clearly equality (b) also holds for all $\Ab_{h}=\{W^{0}(b)\}_h$ 
with $b$ a Fourier multiplier. 
Suppose that $a\in[\PC,\SO, L^\infty_0]$ and $\Ab_{h}=\{aI\}_h\in\Fc^{T}_{h}$. Then
$\Ws^+(\Ab_{h})$ exists, and  in particular, $a_h:=\lim_{n\to\infty}a(h_n)\in\Cb$ exists.
Defining intervals $J_n:=[n/2,3n/2]$, one can choose
a continuous function $\varphi$ which is supported in $J_1$ and $\varphi(1)=\|\varphi\|=1$. 
Then, actually, it holds much more, namely
$\osc(a,J_n)\to 0$ as $n\to\infty$, and hence $\|(a-a_h)\varphi_n\|\leq \|(a-a_h)\chi_{J_n}\|\to 0$. 
Consequently, 
\begin{align*}
\Phi_1(\{aI\}_h+\Jc_{h}^{T})&=\Phi_1(\{a\varphi_n I\}_h+\Jc_{h}^{T})\\
&=\Phi_1(\{a_h\varphi_n I\}_h+\Jc_{h}^{T})+\Phi_1(\{(a-a_h)\varphi_n I\}_h+\Jc_{h}^{T})\\
&=\Phi_1(\{a_hI\}_h+\Jc_{h}^{T})+0.
\end{align*}
Thus, we have (b) for all generators of $\Ac_1$. An arbitrary sequence $\Ab\in\Ac_1$ 
is always the norm limit of a sequence $(\Ab^{(m)})$ consisting of finite linear combinations 
and products of the generators. By a standard diagonal argument one can pass to a
sequence $h\subset g$ such that $\Ab_h$ and all $\Ab^{(m)}_{h}$ belong to $\Fc^T_{h}$.
Since (b) is true for all approximations $\Ab_h^{(m)}$ it must hold for $\Ab_h$ as well.

Lastly, if $|s|>1$ then we benefit from the restriction to the finite section algebra $\Fc_{\Ac_1}$,
apply $\Phi_{s}(\{P_{h_{n}}\}+\Jc_{h}^{T})=\Phi_{s}(\{0\}+\Jc_{h}^{T})$,
and arrive at (even for arbitrary $h$)
\[\Phi_{s}(\{A_{h_{n}}\}+\Jc_{h}^{T})
=\Phi_{s}(\{P_{h_{n}}A_{h_{n}}P_{h_{n}}+\mu_\Ab(I-P_{h_{n}})\}+\Jc_{h}^{T})
=\mu_\Ab\Phi_{s}(\{I\}+\Jc_{h}^{T}).\]

Now let $B$ be a $\Pc$-regularizer of $\Ws^c(\Ab_h)$, i.e. 
$\{B\}\{\Ws^c(\Ab_h)\}=\{I\}+\{K\}$ with a $\Pc$-compact $K$. One easily
checks that this yields the invertibility of all snapshots $\Ws^\pm(\{\Ws^c(\Ab_h)\}_h)$
and $\Hs^s(\{\Ws^c(\Ab_h)\}_h)$, $s\in\Rb$, and that their respective inverses are the snapshots
of $\{B\}_h$. Thus $\{B\}_h\in\Fc^T_h$, even $\{B\}_h\in \Lc^T_h$ due to the inverse closedness
of the algebra $\mathcal{L}_{h}^{T}/\Jc_{h}^{T}$.
The remaining cases are treated similarly.
\end{proof}

Applying the Allan-Douglas principle we obtain

\begin{prop}
Let $\Ab\in\Fc_{\Ac_1}$ and all the snapshots $\Ws^{t}(\Ab_{g})$, $t\in T_{\Pc}$, $g\in\Hb_\Ab$,
be $\Pc$-Fredholm. Then every $T$-structured subsequence of $\Ab$ has a $\Jc^T$-Fredholm subsequence 
$\Ab_h$, i.e. $\Ab_h+\Jc^T_h$ is invertible in $\Fc^T_h/\Jc^T_h$.
\end{prop}
\begin{proof}
Let $\Ab_g$ be $T$-structured. Proposition \ref{localization} yields a subsequence $\Ab_h$.
By the Allan-Douglas local principle, it is enough to show that $\Phi_{s}(\mathbb{A}_{h}+\Jc_{h}^{T})$ 
is invertible for every $s\in\dot{\Rb}$ providing the snapshots $\Ws^{t}(\Ab_{h})$, $t\in T_{\mathcal{P}}$,
are $\mathcal{P}$-Fredholm.

Let $s=1$, then 
$\Phi_{1}(\mathbb{A}_{h}+\Jc_{h}^{T})=\Phi_{1}(\{V_{h_{n}}\Ws^{+}(\Ab_{h})V_{-h_{n}}\}+\Jc_{h}^{T})$,
and the $\Pc$-regularizer $B$ of $\Ws^{+}(\Ab_{h})$ provides an inverse 
$\Phi_{1}(\{V_{h_{n}}BV_{-h_{n}}\}+\Jc_{h}^{T})$ of $\Phi_{1}(\mathbb{A}_{h}+\Jc_{h}^{T})$. 
Analogously, the $\mathcal{P}$-Fredholmness
of $\Ws^{-}(\Ab_{h})$ and $\Ws^{c}(\Ab_{h})$ imply the invertibility
of the local elements at $s=-1$ and $|s|<1$, respectively. Finally,
if $|s|>1$ the $\mathcal{P}$-Fredholmness of $\Ws^{+}(\Ab_{h})$,
which is of the form $\chi_{-}\Ws^{+}(\Ab_{h})\chi_{-}I+\mu_\Ab(1-\chi_{-})I$
implies that $\mu_\Ab$ is non-zero and therefore the claim.
\end{proof}
Summarizing the previous results together with Theorem \ref{TRich}
we obtain 

\begin{thm}
\label{CRich-1} For $\Ab=\{A_{n}\}\in\Fc_{\Ac_1}$ it holds that 
\begin{itemize}
\item If all snapshots of $\Ab$ are Fredholm then $\Ab$ has finite $\alpha$-
and $\beta$-number given by 
\begin{align*}
\alpha(\Ab) & =\max_{h\in\Hb_{\Ab}}\left[\sum_{t=-,c,+}\dim\ker\Ws^{t}(\Ab_{h})+\sum_{t\in\Rb}\dim\ker\Hs^{t}(\Ab_{h})\right],\\
\beta(\Ab) & =\max_{h\in\Hb_{\Ab}}\left[\sum_{t=-,c,+}\dim\coker\Ws^{t}(\Ab_{h})+\sum_{t\in\Rb}\dim\coker\Hs^{t}(\Ab_{h})\right].
\end{align*}
\item If one snapshot is not Fredholm then $\mathbb{A}$ cannot have both
finite $\alpha-$ and $\beta-$ number.
\item $\Ab$ is stable if and only if all snapshots are invertible. 
\end{itemize}
\end{thm}

This particularly yields our motivating Theorem \ref{TIntro}.

\subsection{The flip and Hankel operators}\label{SFlip}
In this very last step we are going to include the flip operator
\[J:L^p(\Rb)\to L^p(\Rb),\; (Jf)(x)=f(-x)\]
and by this get access also to Hankel operators which are, by definition, of the form
$H(b)=\chi_+W^0(b)\chi_-J$. More precisely, we want to consider the algebra
\begin{equation}\label{EBc}\begin{split}
\Bc:=\alge&\left\{ \{J\},\,\{aI\},\,\{W^0(b)\},\,\{K\},\,\{P_n\}:\right. \\
 & \;\;\left. a\in[\PC,\SO,L_0^{\infty}],\; b\in[\PC^\lambda,\SO,\BUC]_p,\;\lambda\in\Rb,
 \;K\in\pc{X}\right\}.
\end{split}\end{equation}
Unfortunately, this is not possible directly by the above approach because $\{J\}$ is not $T$-rich as 
the $\Ws^\pm$- and the $\Hs^t$-snapshots, $t\neq0$, do not exist for any subsequence. 
The point is that the local properties at the ``directions'' $t=+$ and \mbox{$t=-$} (as 
well as at different $t\in\Rb$) are separated in case of $\Ac$-sequences, whereas $\{J\}$ 
now connects the directions $+$ and $-$ (resp. $t$ and $-t$).
Therefore we introduce an adapted symmetrized framework which
combines the snapshots $\Ws^\pm$, as well as the pairs $\Hs^s$, $\Hs^{-s}$ $(s>0)$
into operator-matrix-valued snapshots $\Ws^*$, $\Hs^{s,*}$, respectively. These mappings are 
again algebra homomorphisms $\Fc^S_h\to \lb{L^p(\Rb)\times L^p(\Rb)}$ on certain 
Banach algebras $\Fc_h^S$ of $S$-structured (sub)sequences, which particularly cover all 
$\Ab_h\in\Fc^T_h$ ($T$-structured in the former sense) and $\{J\}$, hence all
generators of $\Bc$, with
\begin{align}\label{EWstar}
\Ws^*(\Ab_h)&=\begin{pmatrix}\Ws^+(\Ab_h)& 0 \\0&\Ws^-(\Ab_h)\end{pmatrix}
& \Ws^*\{J\}&=\begin{pmatrix}0& J \\J&0\end{pmatrix} \\
\notag
\Hs^{s,*}(\Ab_h)&=\begin{pmatrix}\Hs^s(\Ab_h)& 0 \\0&\Hs^{-s}(\Ab_h)\end{pmatrix}
& \Hs^{s,*}\{J\}&=\begin{pmatrix}0& J \\J&0\end{pmatrix} \quad(s>0).
\end{align}
The homomorphisms $\Ws^c$ and $\Hs^0$ remain as before. Then we get for 
the finite section algebra $\Fc_{\Bc}$ that is generated by all
$\{P_nA_nP_n+ (I-P_n)\}$ with $\{A_n\}\in\Bc$	the following.
\begin{thm}\label{TFlip}
For $\Ab=\{A_{n}\}\in\Fc_{\Bc}$ it holds that 
\begin{itemize}
\item If all snapshots of $\Ab$ are Fredholm then $\Ab$ has finite $\alpha$-
and $\beta$-number with
\begin{align*}
\alpha(\Ab) =\max_{h\in\Hb_{\Ab}}\Big[\dim\ker\Ws^{c}(\Ab_{h})&+\dim\ker\Hs^{0}(\Ab_{h})\\
&+\dim\ker\Ws^{*}(\Ab_{h})+\sum_{s>0}\dim\ker\Hs^{s,*}(\Ab_{h})\Big]
\end{align*}
(similarly for $\beta(\Ab)$) and for every $S$-structured subsequence $\Ab_h$
\[\lim_{n\to\infty}\ind A_{h_n} = \ind \Ws^{c}(\Ab_{h}) + \ind \Hs^{0}(\Ab_{h})
+\ind\Ws^{*}(\Ab_{h})+\sum_{s>0}\ind\Hs^{s,*}(\Ab_{h}).\]
\item If one snapshot is not Fredholm then $\mathbb{A}$ cannot have both
finite $\alpha-$, $\beta-$ number.
\item $\Ab$ is stable if and only if all snapshots are invertible. 
\end{itemize}
\end{thm}
We only sketch the proof since it follows in large parts the same line as before. 
The main technical ingredients here are special tricky transformations 
which help to overcome the unpleasant ``non-local'' behavior of the flip operator.

\begin{proof}
\textbf{1st step:}
Given $\Ab=\{A_n\}\in\Rc^T$, the sequence $\{\tilde{A}_n\}$ of operators
\[\tilde{A}_n:=\begin{pmatrix}A_n & 0 \\ 0 & I\end{pmatrix}\in\lb{L^p(\Rb)\times L^p(\Rb)}\]
has the same stability properties, the same $\alpha$- and $\beta$-number, and the 
$\tilde{A}_n$ have the same Fredholm indices as the $A_n$, resp. 
Set $R_1:=\chi_+I$, $R_2:=W^0(\chi_+)$, $S_i:=I-R_i$,
\[\hat{P}_n:= \begin{pmatrix}P_n& 0 \\0&P_n\end{pmatrix} \quad\text{and}\quad
T_i:=\begin{pmatrix}R_i & S_i \\ S_i & R_i\end{pmatrix}\in\lb{L^p(\Rb)\times L^p(\Rb)}
\quad(i=1,2).\]
Then $T_i^{-1}=T_i$ and $\hat{\Pc}=(\hat{P}_n)$ is a uniform approximate identity on 
$\lb{L^p(\Rb)\times L^p(\Rb)}$.

\textbf{2nd step:}
Define the index set $S:=\{c,+\}\cup[0,\infty)$ and let $\Fc^S$ be the set of all 
sequences $\Bb:=\{B_n\}\subset\lb{L^p(\Rb)\times L^p(\Rb)}$ such that the limits
\begin{align}
\Ws^{c,*}(\Bb)&:=\phlimn B_n  & \label{EWcs}\\
\Ws^{+,*}(\Bb)&:=\phlimn\begin{pmatrix}V_{-n}& 0 \\0&V_n\end{pmatrix}T_1B_nT_1\begin{pmatrix}V_{n}& 0 \\0&V_{-n}\end{pmatrix} \label{EWps}
  \\
\Hs^{0,*}(\Bb)&:=\slim\begin{pmatrix}Z_n^{-1}& 0 \\0&Z_n^{-1}\end{pmatrix}B_n \begin{pmatrix}Z_{n}& 0 \\0&Z_{n}\end{pmatrix} \label{EH0s} &\\
\Hs^{s,*}(\Bb)&:=\slim\begin{pmatrix}Z_n^{-1}U_{s}& 0 \\0&Z_n^{-1}U_{-s}\end{pmatrix}T_2B_nT_2\begin{pmatrix}U_{-s}Z_{n}& 0 \\0&U_sZ_{n}\end{pmatrix}
&  \quad(s>0)\label{EHss}
\end{align}
exist. One immediately 
checks that also in this new setting $\Fc^S$ the two conditions (I) and (II) in Section \ref{SStruct}
are fulfilled. Hence the construction of compact $\Jc^S$-sequences, the larger algebra 
$\Rc^S$ of $S$-rich sequences and Theorems \ref{TTstructMain} and \ref{TRich} are available.

\begin{figure}
	\centering
  \includegraphics[scale=1.2]{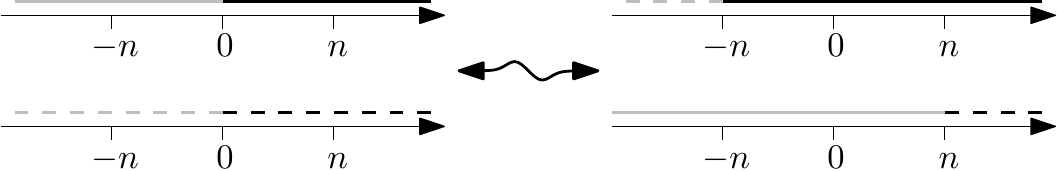}
  \caption{illustration of the transformation in \eqref{EWps}.}
\end{figure}

\textbf{3rd step:}
By straightforward computations we can easily check that for every 
$\Ab_h=\{A_{h_n}\}\in\Fc^T_h$ and also for $\Ab_h=\{J\}_h$ the snapshots 
$\Ws^*(\Ab_h)=\Ws^{+,*}(\tilde{\Ab}_h)$ and 
$\Hs^{s,*}(\Ab_h)=\Hs^{s,*}(\tilde{\Ab}_h)$ are as announced in 
\eqref{EWstar}, and that
\[\Ws^{c,*}(\tilde{\Ab}_h)=\begin{pmatrix}\Ws^c(\Ab_h)& 0 \\0&I\end{pmatrix},\quad
\Hs^{0,*}(\tilde{\Ab}_h)=\begin{pmatrix}\Hs^0(\Ab_h)& 0 \\0&I\end{pmatrix}.\]
Indeed, \eqref{EWcs}, \eqref{EH0s} are clear and \eqref{EWps} 
\[
\phlimn\begin{pmatrix}V_{-h_n}(\chi_+A_{h_n}\chi_++\chi_-)V_{h_n}& V_{-{h_n}}\chi_+A_{h_n}\chi_-V_{-{h_n}}\\
V_{h_n}\chi_-A_{h_n}\chi_+V_{h_n} & V_{h_n}(\chi_-A_{h_n}\chi_-+\chi_+)V_{-{h_n}}\end{pmatrix}
\]
exists by the following simple facts: $\plim V_{\mp n}\chi_{\pm}V_{\pm n}=I$, 
\mbox{$\plim V_{ \mp n}\chi_{\mp}V_{\pm n}=0$,}
$J\chi_{\pm}V_{\pm n}=\chi_{\mp}V_{\mp n}J$ as well as $V_n=V_{2n}V_{-n}=V_{-n}V_{2n}$ etc. 

Finally, \eqref{EHss} has the very similar structure
\[\slim\begin{pmatrix}Z_{h_n}^{-1}U_{s}(R_2A_{h_n}R_2+S_2)U_{-s}Z_{h_n}& 
Z_{h_n}^{-1}U_{s}R_2A_{h_n}S_2U_sZ_{h_n}\\
Z_{h_n}^{-1}U_{-s}S_2A_{h_n}R_2U_{-s}Z_{h_n} & 
Z_{h_n}^{-1}U_{-s}(S_2A_{h_n}S_2+R_2)U_sZ_{h_n}\end{pmatrix}\]
where $\slimnn Z_n^{-1}U_{s}R_2U_{-s}Z_n=\slimnn Z_n^{-1}U_{-s}S_2U_sZ_n=I$, 
$\slimnn Z_n^{-1}U_{s}S_2U_{-s}Z_n=\slimnn Z_n^{-1}U_{-s}R_2U_sZ_n=0$, 
$JS_2U_sZ_n=R_2U_{-s}Z_nJ$ and $JR_2U_{-s}Z_n=S_2U_sZ_nJ$. 
The convergence for $\{A_{h_n}\}\in\Fc^T_h$ of e.g. the upper right corner follows by reformulation
\[Z_{h_n}^{-1}U_{2s}Z_{h_n}\cdot Z_{h_n}^{-1}U_{-s}R_2U_sZ_{h_n}\cdot
Z_{h_n}^{-1}U_{-s}A_{h_n}U_sZ_{h_n}\cdot Z_{h_n}^{-1}U_{-s}S_2U_sZ_{h_n}
\to 0\cdot \Hs^{-s}\{A_{h_n}\}\cdot I.\]
Then, by this and the 1st step, the formulas in Theorems \ref{TTstructMain} 
and \ref{TRich} easily translate to the assertion of the present Theorem \ref{TFlip}.

\textbf{4th step:}
All we are left with is to remove the $\Jc^S$-Fredholm conditions in the general 
Theorems \ref{TTstructMain} and \ref{TRich}, again by localization as in Section \ref{SLocal}.
Obviously, 
\[\Dc:=\left\{\left\{\begin{pmatrix}\varphi_n I& 0 \\0&\varphi_n I\end{pmatrix}\right\}:
\varphi\in C(\dot{\Rb}), \varphi=-\varphi\right\}\]
is a closed commutative subalgebra of $\Fc^S\subset\Rc^S$, and $\Dc$ as well as
all $(\Dc_h+\Jc^S_h)/\Jc^S_h$ are isometrically isomorphic to $C[0,\infty)$.
Furthermore $[\Cb,\tilde{\Ab}]\in\Jc^S$ for every $\Cb\in\Dc$ and every $\Ab\in\Fc_{\Bc}$,
by application of Corollary \ref{CComm1} and the relation $[J,\varphi_n]=0$.
Then, one still identifies the local cosets of $\tilde{\Ab}_h$, with 
$\{\Ws^{c,*}(\tilde{\Ab}_h)\}$ at the points $s\in[0,1)$ and with a certain 
multiple of the identity at the points $s>1$. 
At the point $s=1$ the new snapshot $\Ws^{+,*}(\tilde{\Ab}_h)$ does the job, 
and Allan-Douglas finishes the proof.
\end{proof}

\section{Conclusions}
The approach of this paper provides not only the well known results on the stability / 
applicability
of the Finite Section Method for classes of convolution type operators with 
piecewise continuous and slowly oscillating data \cite{KMS2010,RSS2010a}, but goes far
beyond: On the one hand it completes the picture by describing also non-stable
sequences in terms of the asymptotic behavior of their approximation numbers and 
Fredholm indices. On the other hand it covers much larger classes of data, including
almost periodic, even $\BUC,$ multipliers $b$,  
as well as operators of multiplication 
with almost arbitrary functions $a$ having restrictions only on the behavior
at infinity. Actually, our approach permits to build upon much simpler (and 
probably even more general) definitions of the multiplier algebras. Finally, 
also the flip hence Hankel operators are included.

Moreover, the constructions and proofs have become much simpler. In the previous
works the localization had to be done over more involved commutative subalgebras
of sequences arising from slowly oscillating functions whose definition is technical 
and the characterization of their maximal ideal spaces required hard proofs.
This is now replaced by the simple algebra $\Cc$. The keys for this simplification
are the application of the $\Pc$-theory, the concept of rich sequences, and the
valuable characterizations of quasi-banded operators.

Another remarkable point is that the role of the snapshots changed dramatically.
While in the previous literature the $\Ws$-directions served for the construction 
a sufficiently large ideal $\Jc$ such that the respective quotient has a large
center, and the $\Hs$-snapshots provided the desired representatives for
the identification of the local cosets, we now turn the table.
In the present paper only the $\Ws$-snapshots are used in the identification of
the local cosets, whereas the $\Hs$-directions provide the appropriate $\Jc^T$.


\def\cprime{$'$} \def\cprime{$'$}

\end{document}